\documentclass{article}
\usepackage[top=4cm,left=3.5cm,right=2.5cm]{geometry}
\usepackage{amssymb}
\usepackage{amsmath}
\usepackage{amsthm}
\usepackage{graphicx}
\usepackage{varwidth}
\usepackage{float}
\usepackage{array}
\usepackage{subfig}
\usepackage{grffile}
\newcommand{\R}{\mathbb{R}}
\newcommand{\norm}[1]{\Vert #1 \Vert}

\usepackage{color}
\newtheorem{theorem}{Theorem}[section]
\newtheorem{proposition}[theorem]{Proposition}
\newtheorem{lemma}[theorem]{Lemma}
\newtheorem{corollary}[theorem]{Corollary}
\newtheorem{remark}[theorem]{Remark}
\newtheorem{example}{Example}

\title{Color Bregman TV}
\author{Michael Moeller, Eva-Maria Brinkmann, Martin Burger, Tamara Seybold}
\begin{document}
\maketitle
\begin{abstract}
In this paper we present a novel iterative procedure for multichannel image and data reconstruction using Bregman distances. The motivation for our approach is that in many application multiple channels share a common subgradient with respect to a suitable regularization. This implies desirable properties such as a common edge set (and a common direction of the normals to the level lines) in the case of the total variation (TV). Therefore, we propose to determine each iterate by regularizing each channel with a weighted linear combination of Bregman distances to all other image channels from the previous iteration. In this sense we generalize the Bregman iteration proposed by Osher et al. in \cite{osh-bur-gol-xu-yin} to multichannel images. We prove the convergence of the proposed scheme, analyze stationary points and present numerical experiments on color image denoising, which show the superior behavior of our approach in comparison to TV, TV with Bregman iterations on each channel separately, and vectorial TV. Further numerical experiments include image deblurring and image inpainting. Additionally, we propose to use the infimal convolution of Bregman distances to different channels from the previous iteration to obtain the independence of the sign and hence the independence of the direction of the edge. While this work focuses on TV regularization, the proposed scheme can potentially improve any variational multichannel reconstruction method with a one-homogeneous regularization.
\end{abstract}

\vspace{0.3cm}


\vspace{0.3cm}

\noindent \textbf{AMS:} 47A52, 65J22, 68U10, 49M30
	
\vspace{0.2cm}
 
\section{Introduction}
\label{sec:Introduction}
Variational methods for image restoration have become one of the dominant approaches due to their flexibility and wide range of applicability. A majority of image processing and reconstruction tasks have been modeled as energy minimization problems in which the reconstructed image $\hat{u}$ is determined by 
\begin{align}
\hat{u} = \arg \min_u H_f(u) + J(u), \label{model0}
\end{align} 
for a data fidelity term $H_f(u)$ measuring the closeness to the data $f$ and a regularization term $J(u)$ imposing some kind of smoothness or prior information (cf. \cite{Aub06,tvzoo1,chanshen,2009-FAIR}). Classical choices for the fidelity term are squared norms between $f$ and some operator applied to the image $u$ or log-likelihoods derived from different noise models (cf. \cite{SBMB09, Aub08, Le07,Saw13}). The most frequently used regularization terms are total variation (cf. \cite{ROF,tvzoo1}), variants thereof (cf. e.g. \cite{Bre10,Gil09}), or other $\ell^1$-type regularizations yielding sparse reconstructions in certain dictionaries (cf. \cite{Daubechies20031,kutyniok2012shearlets,starckcandes}). Further improvements in variational image restoration have been made by considering Bregman iterations related to \eqref{model0}, iteratively improving the reconstruction via
\begin{align}
u^{k+1} = \arg \min_u H_f(u) + J(u) - J(u^k) - \langle p^k, u - u^k \rangle, \label{Bregman}
\end{align} 
where $p^k$ is in the subdifferential of $J$ at $u^k$, $p^k \in \partial J(u^k)$. Here the regularization term is a generalized Bregman distance between $u$ and $u^k$. Starting with an initial value such that $p^0= 0 \in \partial J(u^0)$, \eqref{model0} is exactly the first step of the Bregman iteration. In the further step the regularization is replaced from $J$, i.e. the Bregman distance to $u^0 =0$, to the Bregman distance to the last image. A regularizing effect in case of noisy data can be obtained by appropriately terminating the iterative procedure (cf. \cite{osh-bur-gol-xu-yin}). Both experimental (cf. \cite{osh-bur-gol-xu-yin,benningbrunemueller,primaldualbregman,marquina,zhang}) as well as analytical results (cf. \cite{osh-bur-gol-xu-yin, bur07, benningbrunemueller, benningBurger, gilboa, Xu07, moel12}) confirm the superior properties of Bregman iterations, in particular the systematic error compared to \eqref{model0} is reduced.

The majority of the investigations with methods mentioned above was related to gray-scale images, in particular Bregman iterations have been investigated exclusively to that case (with the trivial exception of separately applying Bregman iterations to different color channels).
In the case of color - or, more generally, multichannel - image restoration, one crucial question is what kind of relation between the different image channels can be assumed. Even among methods considering the total variation (TV) as a regularization functional, many different variants have been proposed, e.g. Vectorial TV \cite{Blo96, Sap96, Cre10} or TV in different color spaces, e.g. chrominance and luminance based TV in \cite{Con12}. While these regularizations work well in many applications, they share the property of coupling the intensities of different color channels. Our novel idea is to use Bregman distance regularization in an iterative setting to make use of the fact that the edge sets of different image channels should coincide. 

Consider an underdetermined multichannel image restoration problem, which requires the solution of 
\begin{equation}
	K_i u_i = f_i
\end{equation}
for $M$ different channels $i \in \{1, \ldots ,M\}$ with measurements $f_i$ and bounded linear operators $K_i$. Our motivation is to look for a solution $u=(u_1, \ldots ,u_M)$ with reasonably small total variation in each channel and additionally require the channels $u_i$ to share a common edge set. This way, no assumption has to be made on the sizes of the jumps across the edges in each channel. As we will see in the course of this paper, a common edge set is strongly correlated with a common subgradient of the total variation. Hence, in its strictest form, one could consider minimizing 
\begin{equation}
	\|K_i u_i - f_i\|^2
\end{equation}
under the condition that there exists a $p \in \partial J(u_i)$ for all $i=1,\ldots,M$, i.e. that the channels share a common subgradient and consequently the edge set in the case $J=TV$. Because the constraint of a common subgradient is difficult to impose directly, we will use the Bregman distances between different color channels in an iterative procedure to measure the  differences between edge sets of different channels.  More precisely, we shall compute updates from using an average over Bregman distances to different channels,  i.e.
\begin{equation}
	u_i^{k+1} \in \text{arg}\min_{u_i \in {\cal U}} \left( \frac{\lambda}{2} \Vert K_i u_i  -f_i \Vert^2 + \sum_{j=1}^M w_{i,j} (J(u_i) - \langle p_j^k, u_i \rangle) \right)
	\label{eq:colorBregman}
\end{equation}
with positive weights $w_{i,j}$,  $p_j^k \in \partial J(u_j^k)$, and a data fidelity weight $\lambda$. Clearly this makes sense in particular for one-homogeneous $J$, and we will focus on the particular case of $J=TV$.
Note that for a weight matrix being identical we obtain the channel-wise Bregman iteration. In this sense, our approach generalizes and extends the work of Osher et al. \cite{osh-bur-gol-xu-yin} for color image restoration.

While a common subgradient of different image channels forces the edges to - loosely speaking - `go into the same direction', e.g. in 1d either all be a `step up' or all be a `step down', we also consider the case of different channels sharing an edge set but not (necessarily) the edge direction. This will be particularly useful for image channels coming from different image modalities, e.g. in combined medical imaging devices such as SPECT/CT or PET/MR scanners. Mathematically, we use the infimal convolution of Bregman distances to achieve the independence of the edge direction. 

Besides TV, collaborative sparse reconstruction using $\ell^1$ regularization is another possible application and motivation for our approach. In case multiple sparse quantities have to be reconstructed from multiple measurements and one can expect the solutions to have the same sign, one can use the Bregman distances between the quantities as a penalty, thus encouraging a common subgradient. The infimal convolution of Bregman distances allows to relax this constraint. The variables are encouraged to have the same support, but do not necessarily have to have the same sign. 

The idea of using edge information from different channels or images has been considered in literature before, see for instance in \cite{Tsch03, ehr12, Bal03, Pra2013} and the references therein, but to the best knowledge of the authors no work has been done in the context of jointly minimizing the Bregman distances and hence encouraging common subgradients of different images.

The remainder of the paper is organized as follows: In section \ref{sec:nutshell} we illustrate our general idea using the example of discrete one-dimensional  TV denoising. We continue in section \ref{sec:bregpriors} by discussing Bregman distance priors with a particular focus on the relation between the TV Bregman distance and common edge sets as well as the $\ell^1$ Bregman distance and a common support. Our novel idea for iterative color image restoration using Bregman distances is presented in section \ref{sec:colBreg} along with an analysis of stationary points, the primal-dual and dual formulation, as well as the proof of the well-definedness of the proposed iteration. Section \ref{sec:convergenceanalysis} analyzes the convergence of the proposed scheme. The numerical implementation as well as several color image reconstruction results are presented in section \ref{sec:numerics}. 
Finally, we draw conclusions and outline possible extensions of the proposed scheme in section \ref{sec:conclusions}.

Concerning notation, we shall use the following setting uniformly throughout the paper
\begin{itemize}

\item A subscript $i$ will denote the scalar image at channel $i$, $u=(u_1,\ldots,u_M)$ is the complete multichannel image.

\item A superscript $k$ will always denote the iteration number, i.e. $u^k$ is the $k$-th Bregman iteration, $u_i^k$ its $i$-th channel.

\item The argument will always denote the spatial location, either continuously as $u(x)$ or referring to discrete pixel / voxel number $u(j)$, which will be clear from the context. 
\end{itemize}

\section{The Idea in a Nutshell: Discrete 1d TV}
\label{sec:nutshell}
As a first motivation, let us consider discrete total variation regularization in 1d, which is easy to analyze and clearly motivates our approach. Although the discussion in this section is limited to the case of denoising, we would like to point out that the entire idea easily generalizes to any data fidelity term and can therefore be applied to a wide variety of image reconstruction problems.

 Let $f_i \in \R^N$ be measured signals (later: channels of an image) with $i \in \{ 1, ..., M \}$. Let $D$ denote the finite difference matrix for approximating the gradient. The discrete 1d total variation is defined as $ J(u) = \|Du \|_1 $. The corresponding subdifferential can be characterized as
\begin{equation} \partial J(u) = \{p~|~ p = D^T q, ~ q \in \text{sign}(Du) \} .
\end{equation}
In other words, the $j$-th component of $q$ is given as  
\begin{align*}
q(j) \left \{ \begin{array}{lc} =1& \text{if } (Du)(j) >0, \\
								=-1& \text{if } (Du)(j) < 0, \\
								\in [-1, 1 ] & \text{if } (Du)(j) = 0.
\end{array} \right.
\end{align*}
Now consider the regular Bregman iteration on each signal $f_i$ separately. In the $(k+1)$-th iteration one solves
\begin{align*}
u_i^{k+1} 
&= \arg \min_u \left( \frac{\lambda}{2} \| u_i - f_i\|^2 + \|Du_i\|_1 - \langle q_i^{k}, Du_i \rangle \right), \\
&= \arg \min_u \left( \frac{\lambda}{2} \| u_i - f_i\|^2 + \sum_j \left(\text{sign}(Du_i(j))-q^{k}_i(j) \right) Du_i(j) \right),
\end{align*}
for $D^T q_i^k \in \partial J(u_i^k)$. We can see that the Bregman iteration basically weights the penalization of each component of $Du_i$, or, in other words, weights the penalization of the approximated derivative. If the $j$-th component of $u_i^{k}$ had a `jump up', i.e. $Du_i^{k}(j)>0$, i.e. $q_i^{k}(j)=1$, then a `jump up' of $u_i^{k+1}$ at the $j$-th component is not penalized anymore and the size of the jump is determined exclusively by the minimization of $\| u_i - f_i\|^2 $. Note that this property has been studied in literature and it has been shown theoretically as well as experimentally that Bregman iteration restores the contrast of the reconstructed image and often leads to higher quality results than plain TV minimization. 

In the above Bregman iteration the $q^{k}_i(j)$ could be interpreted as the likelihood of the true image having an edge in the corresponding direction at this component. The higher the value of $q^{k}_i(j)$, the smaller is the penalty for an edge with $Du^{k+1}_i(j)>0$ or, in other words, the more likely is an edge in this direction.  Notice that once $q^{k}_i(j)=1$ or $q^{k}_i(j)=-1$ the magnitude of $Du^{k+1}_i(j)$ does not matter anymore.

 Viewing the $q^{k}_i(j)$ as likelihoods of the true image having an edge, motivates our idea for color Bregman iteration: If we have a signal (e.g. an image) with multiple channels, we consider it to be likely that all channels share the position of the true edges as well as their direction. If $q^{k}_i(j) >0$ for all channels $i$ we want to reduce the penalty on all $u_i^{k+1}$ for having a `jump up' at the $j$-th component. On the other hand, if some $q^{k}_i(j) >0$ and some $q^{k}_l(j) <0$ this behavior is likely due to noise and one should not change the penalization of the corresponding component. We therefore propose to compute the `likelihood' of an edge in each color channel as a linear combination of all $q^{k}_i$, i.e. computing
 \begin{align}
u^{k+1} &= \arg \min_u \left( \frac{\lambda}{2} \| u - f\|^2 + \|Du\|_1 - \langle \tilde{q}_i^{k}, Du \rangle \right),
\end{align}
with $\tilde{q}_i^{k} = \sum_l w_{i,l} q^{k}_l$ and nonnegative weights $w_{i,l}$. We will demonstrate in this paper that this kind of color Bregman iteration has significant advantages over the channel by channel reconstruction of color images. 

As a first example, consider the simple, one-dimensional case shown in figure \ref{fig:1dex}: The first row shows the clean data in the form of three one-dimensional channels shown in red, green and blue, as well as a noisy version of the signal. Note that the heavy noise in comparison to the small signal in the blue channel makes the reconstruction of the blue signal without additional information basically impossible. This claim is underlined by looking at the variable $q$ after the first iteration: The variable $q$ reaching an absolute value of 1 indicates an edge at the corresponding position. More precisely, $q(i)=1$ means a step up and $q(i)=-1$ means a step down at position $i$. 
\begin{figure}[H]
\begin{center}
\begin{minipage}[t]{0.48\textwidth}
\begin{center}
\includegraphics[width=0.666\textwidth, natwidth=560,natheight=420]{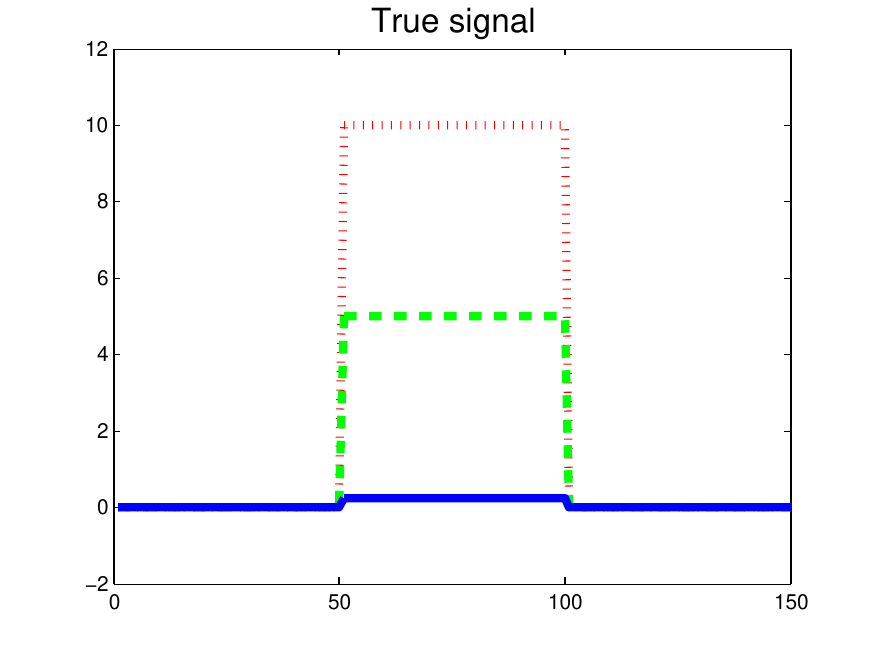}
\end{center}
\end{minipage}
\begin{minipage}[t]{0.48\textwidth}
\begin{center}
\includegraphics[width=0.666\textwidth, natwidth=560,natheight=420]{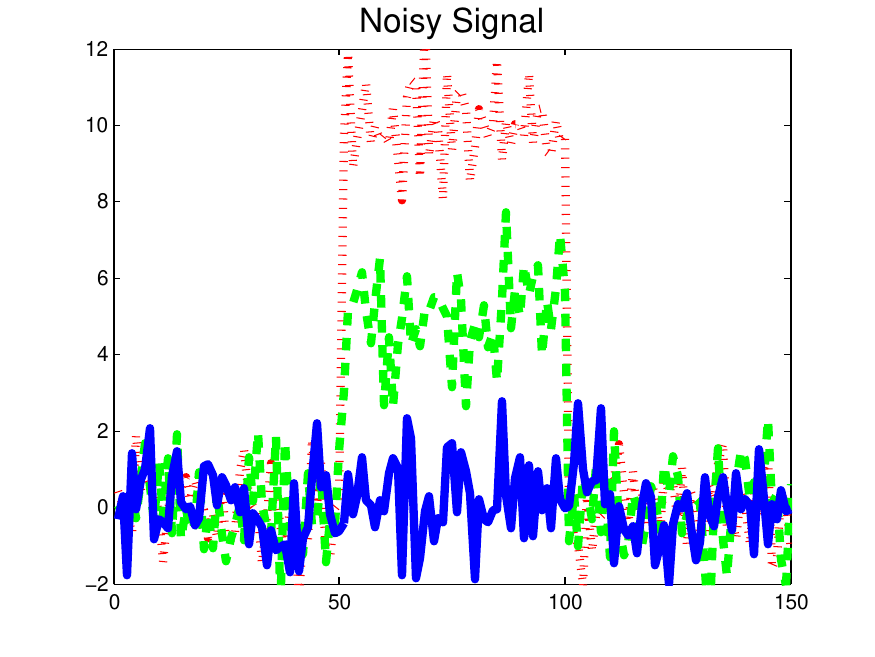}
\end{center}
\end{minipage}
\\
\begin{minipage}[t]{0.48\textwidth}
\vspace{-0.5cm}
\begin{center} Clean signal (three channels)
\end{center}
\end{minipage}
\begin{minipage}[t]{0.48\textwidth}
\vspace{-0.5cm}
\begin{center} Noisy signal (three channels)
\end{center}
\end{minipage}
\\

\vspace{0.3cm}

\begin{minipage}[t]{0.32\textwidth}
\begin{center} \includegraphics[width=1\textwidth, natwidth=560,natheight=420]{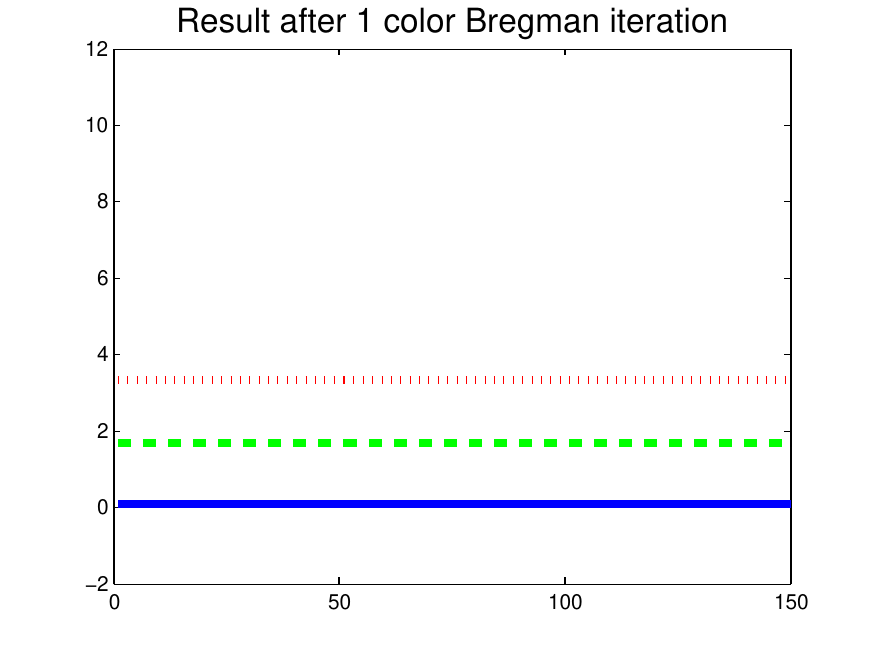}
\end{center}
\end{minipage}
\begin{minipage}[t]{0.32\textwidth}
\begin{center} \includegraphics[width=1\textwidth, natwidth=560,natheight=420]{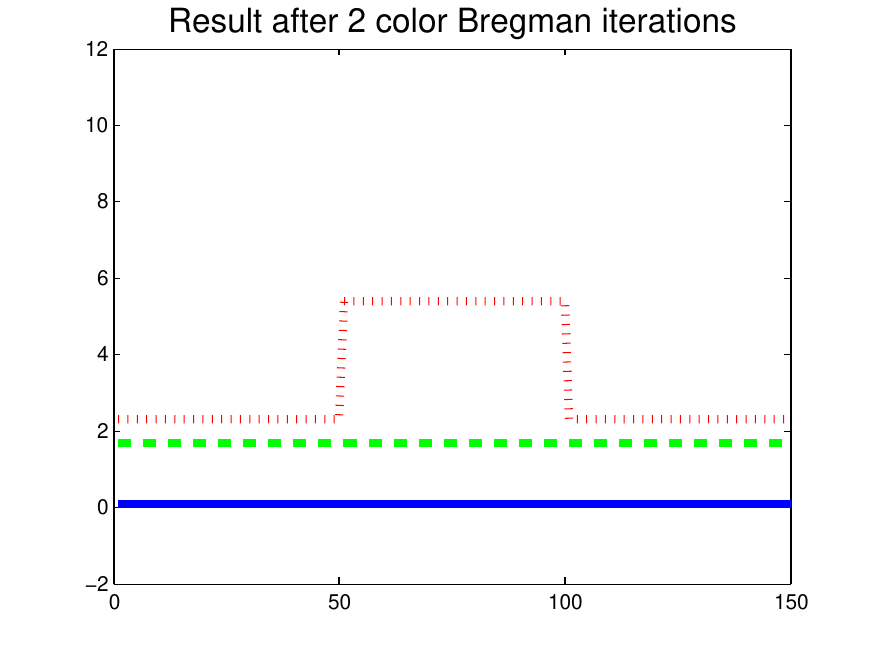}
\end{center}
\end{minipage}
\begin{minipage}[t]{0.32\textwidth}
\begin{center} \includegraphics[width=1\textwidth, natwidth=560,natheight=420]{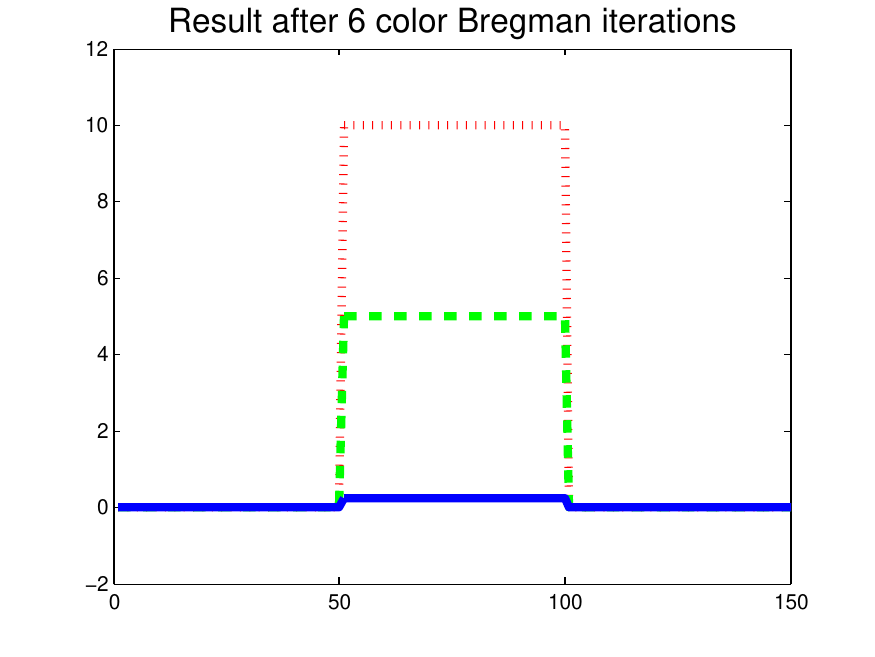}
\end{center}
\end{minipage}
\begin{minipage}[t]{0.32\textwidth}
\vspace{-0.5cm}
\begin{center}  Reconstruction after 1 iteration
\end{center}
\end{minipage}
\begin{minipage}[t]{0.32\textwidth}
\vspace{-0.5cm}
\begin{center} Reconstruction after 2 iterations
\end{center}
\end{minipage}
\begin{minipage}[t]{0.32\textwidth}
\vspace{-0.5cm}
\begin{center} Reconstruction after 6 iterations
\end{center}
\end{minipage}
\\

\vspace{0.3cm}

\begin{minipage}[t]{0.32\textwidth}
\begin{center} \includegraphics[width=1\textwidth, natwidth=560,natheight=420]{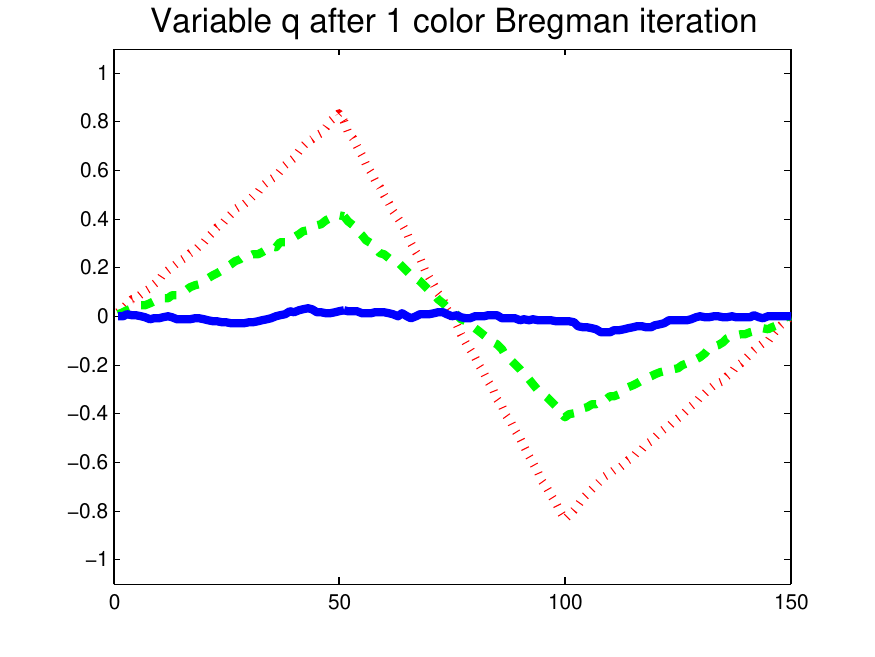}
\end{center}
\end{minipage}
\begin{minipage}[t]{0.32\textwidth}
\begin{center} \includegraphics[width=1\textwidth, natwidth=560,natheight=420]{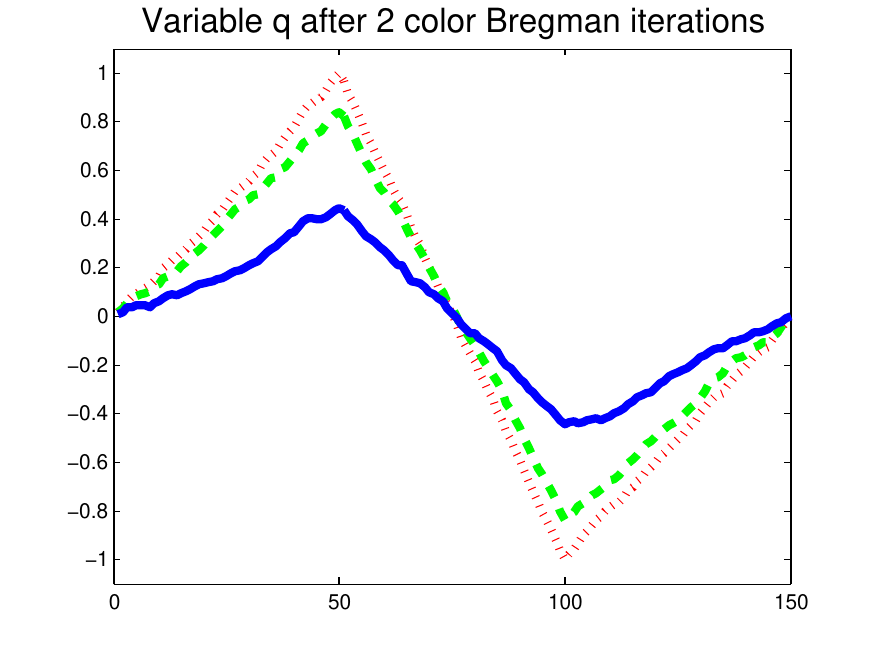}
\end{center}
\end{minipage}
\begin{minipage}[t]{0.32\textwidth}
\begin{center} \includegraphics[width=1\textwidth, natwidth=560,natheight=420]{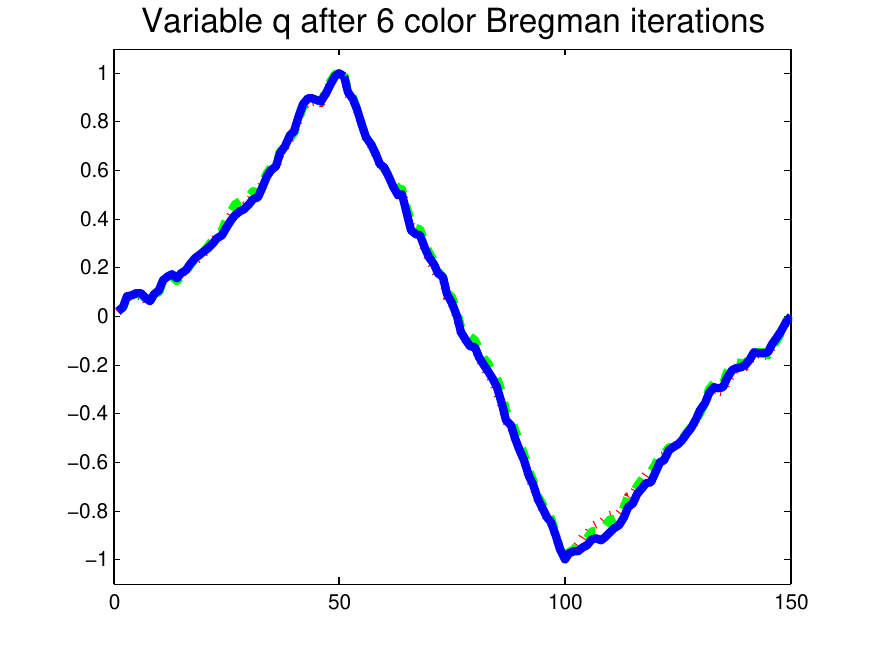}
\end{center}
\end{minipage}
\begin{minipage}[t]{0.32\textwidth}
\vspace{-0.5cm}
\begin{center} Variable $q$ after 1 iteration
\end{center}
\end{minipage}
\begin{minipage}[t]{0.32\textwidth}
\vspace{-0.5cm}
\begin{center} Variable $q$ after 2 iterations
\end{center}
\end{minipage}
\begin{minipage}[t]{0.32\textwidth}
\vspace{-0.5cm}
\begin{center} Variable $q$ after 6 iterations
\end{center}
\end{minipage}
\caption{One-dimensional example of color Bregman iteration}
\label{fig:1dex}
\end{center}
\end{figure}

We can see that the true steps are visible in the red and green but not in the blue channel. Even worse, the minimum and maximum values of $q$ in the blue channel are not at positions $50$ and $100$ such that the first edges introduced in a single channel reconstruction of blue will not be in the right place. As we can see in figure \ref{fig:1dex}, color Bregman iteration allows the blue channel to use the edge information of red and green to obtain a much better estimate for the true variable $q$. Finally, after six iterations, we reconstruct the true signal almost perfectly despite the heavy noise. 

Additional to the color Bregman iteration described above, we look at the case where different images or channels are likely to share an edge set but not necessarily the sign of $Du(j)$. In this case it would be natural to consider
\begin{align}
u_i^{k+1} &= \arg \min \frac{\lambda}{2} \| u_i - f_i\|^2 + \sum_j \left(1- \sum_l w_{i,l}|q^{k}_l(j)| \right) |Du_i(j)| ,
\end{align}
such that the magnitude but not the sign of the ${q}^{k}_l(j)$ matter. Again, we can combine edge information of multiple channels by using a weighted average of the absolute values of the $q^{k}_i(j)$. In a more general setting, we propose to use the infimal convolution of Bregman distances to achieve the independence of the sign. 

In the following sections we will analyze the above regularization procedure in detail in the continuous case and look at different regularizations $J(u)$. We will start by establishing the connection between the Bregman distance and common subgradients of different elements. 

\section{Bregman Distance Priors}
\label{sec:bregpriors}

The Bregman distance as first considered in \cite{bregman} and later generalized to arbitrary proper convex functions $J(u)$ in \cite{osh-bur-gol-xu-yin} is given by
\begin{equation}
D^p_J(v,u) = J(v) - J(u) - \langle p, v-u \rangle
\end{equation}
for $p \in \partial J(u)$. In the last years the Bregman distance has been used for various tasks including regularization \cite{osh-bur-gol-xu-yin}, optimization \cite{splitbregman, cai1, yin2} as well as for error estimates \cite{bur07}. 

For one-homogeneous functionals, the pairs $(u,v)$ with a Bregman distance of zero can be characterized as follows:
\begin{proposition}
\label{prop:zerobregdist}
Let $J$ be one-homogeneous and convex. Then 
\begin{equation}
	p \in \partial J(u) \cap \partial J(v) \Leftrightarrow D_J^p(v,u) = 0.
\end{equation}
\end{proposition}
\vspace*{-12pt}
\begin{proof}
Note that the subdifferential of one-homogeneous functionals $J(u)$ can be classified as
\begin{equation}
\label{eq:onehomogesubdiff}
\partial J(u) = \{p\ | \ \langle p, u \rangle = J(u), \  J(v) - \langle p, v \rangle \geq 0 \; \forall v \},
\end{equation} such that the Bregman distance simplifies to  
\begin{equation}
\label{eq:onehomobregdist}
D^p_J(v,u) = J(v) - \langle p, v \rangle.
\end{equation}
For $p \in  \partial J(v)$ we can use \eqref{eq:onehomogesubdiff} in \eqref{eq:onehomobregdist} to see that $D_J^p(v,u) = 0$. Vice versa, $D_J^p(v,u) = 0$ means that $J(v) = \langle p, v \rangle$ and since $p \in  \partial J(u)$ we have $J(z) - \langle p, z \rangle \geq 0 \ \forall z $ such that $p \in  \partial J(v)$ based on \eqref{eq:onehomogesubdiff}.
\end{proof}

Proposition \ref{prop:zerobregdist} gives a first indication that the Bregman distance measures the distance between the subdifferentials of two elements rather than measuring a direct distance. The motivation for our new, extended Bregman iteration comes from considering the total variation as a regularization functional, for which the subgradients encode information about the edges of an element without considering the size of the jump. As we will see in the next subsection, the corresponding Bregman distance measures the difference between the edge sets of two functions. To our minds the idea of assuming common edge sets for instance for the color channels of an image is very reasonable and allows us to propose a regularization that does not rely on any relation between the values of the color channels. 

We also mention that the assumption of merely equal discontinuity sets can be derived from the usual model of
color images as projections of a hyperspectral image $U$, i.e.
\begin{equation}
	u_i(x) = \int U(x,\omega) \phi_i(\omega) ~d\omega,
\end{equation}
where $\phi_i$ is an appropriate kernel for each channel. Under the natural condition that $U$ is smooth in the spectral dimension $\omega$, but changes between the characteristic spectra of different materials at some discontinuity set $\Gamma$ it is straight-forward to see that the discontinuity set of all images $u_i$ is contained in $\Gamma$ and it is even very likely to be equal to $\Gamma$ for all $i$.

\subsection{TV Bregman Distance and Edges}
Throughout this subsection, we will assume that $J$ is the total variation defined as
$$ J(u) = \sup_{q \in C_0^\infty(\Omega, \R^n),\  \| q \|_{\infty} \leq 1} \int_\Omega ~ \text{div}(q) \ u ~ dx $$
for an image domain $\Omega$.

A well-known result (cf. \cite{Res06}) serves as a first indication that a common edge set can be encoded by joint subgradients respectively by zero Bregman distance:
\begin{lemma}
\label{lem:monotoneBregDistZero}
Let $v=f(u)$ for some monotone function $f$ and $J=TV$. Then 
\begin{equation}
	D_J^p(v,u) = 0 \qquad \forall~p \in \partial J(u). 
\end{equation}
\end{lemma}
\vspace*{-12pt}

Another relevant result is related to images having a joint edge set and some disjoint parts at nonzero distance:
\begin{theorem}
\label{thm:commonedges}
Let $u$ and $v$ be piecewise constant with  $C^1$ discontinuity sets
\begin{align}
E(u) &= \Gamma \cup \Sigma, \\
E(v) &= \Gamma \cup \Upsilon, 
\end{align}
such that the shared edge set $\Gamma$, the respective disjoint parts of the edge sets $\Sigma$ and $\Upsilon$, and the domain boundary $\partial \Omega$ have pairwise positive distance and are boundaries of open sets, respectively. Moreover, let
\begin{equation}
\label{eq:samedirection}
 [u] [v] > 0 \quad \text{on} \ \Gamma, 
\end{equation} 
where $[u]$ denotes the jump across $\Gamma$ in normal direction. 
Then there exists $p \in \partial J(u) \cap \partial J(v)$.
\end{theorem}
\begin{proof}
Let $\epsilon$ be smaller than the minimal distance between any of the sets $\Gamma$, $\Sigma$, $\Upsilon$, $\partial \Omega$. Let $d$ denote the signed distance function to $E(u) \cup E(v)$. 
Moreover, let $f_\epsilon$ be a nonnegative function with support in $(-\epsilon,\epsilon)$ attaining its maximum value one at argument zero. Then 
$$ p = \nabla \cdot ( s f_\epsilon(d) \nabla d), $$
is a subgradient if $s$ equals the sign of $[u]$ on $E(u)$ and the sign of $[v]$ on $E(v)$. Note that the latter choice is possible since $[u][v] > 0$ on the intersection of the jump sets. We only verify that $p \in \partial J(u)$, the assertion for $v$ is completely analogous. First of all, we know that for a piecewise constant function with smooth subset the total variation is given by (cf. \cite{ambrosio,evansgariepy})
$$ J(u) = \int_{E(u)} |[u]|~d\sigma. $$
Now we have that $p$ is the divergence of a vector field with maximal value of the Euclidean norm equal to one and compact support, which implies
$$ \langle p, \varphi \rangle \leq J(\varphi) \qquad \forall~\varphi \in BV(\Omega). $$
It remains to verify $\langle p, u  \rangle = J(u)$. We have
$$ \langle p, u \rangle = \int_\Omega u \nabla \cdot (s f_\epsilon(d) \nabla d) ~dx. $$ 
Integrating by parts on the open sets
separated by $E(u)$ and using that $\nabla u$ vanishes away from the discontinuity set we obtain
$$ \langle p, u \rangle = \int_{E(u)} [u] s f_\epsilon(d) \nabla d \cdot n ~d\sigma. $$
Since on $E(u)$ we have $d=0$, $\nabla d=n$, and $ [u]s = |[u]|$, we conclude that the right-hand side coincides with the above formula for the total variation of $u$. 
\end{proof}

Theorem \ref{thm:commonedges} confirms that a joint subgradient indeed exists (respectively, two elements have a zero Bregman distance) if certain conditions on the edge sets are met. We can see that one image having an edge at a place where the other one is constant does not influence the existence of a joint subgradient. This is an important piece of information for using the Bregman distance as a regularization, because it tells us that no edges are introduced artificially if one channel has an edge at a place where the other one does not.

Theorem \ref{thm:commonedges}, however, also tells us that a common edge alone is not enough. Condition \eqref{eq:samedirection} shows that on the common edge set, the jumps across the discontinuity sets have to point in the same direction. In the following we will speak of the jump across the discontinuity set as the `direction' of an edge. In this sense, a zero Bregman distance encodes that the edges two images (or channels) have in common `point in the same direction'. Note that the condition \eqref{eq:samedirection} is the two-dimensional analogon of the $Du_i(j)$ having the same sign for all $i$ in section \ref{sec:nutshell}. Also note that the edge-aligning property of the Bregman distance, or more precisely the alignment of the level lines, was one of the motivations for the iterative Bregman regularization in \cite{osh-bur-gol-xu-yin}. 

In the next subsection we will look at $\ell^1$ regularization, where the analogous property of the `edges pointing in the same direction' is easy to understand.

\subsection{$\ell^1$ Bregman Distance and Support}

It is well known that the subdifferential of the $\ell^1$ norm can be characterized as follows:
\begin{align}
p \in \partial \|u\|_1 \ \Leftrightarrow \ p(j) \left\{ \begin{array}{ll} = 1 & \text{if } u(j) >0, \\ =-1 & \text{if } u(j) <0, \\ \in [-1,1] & \text{else.}
\end{array} \right.
\end{align}
With the help of this representation it is easy state the analogous version of Theorem \ref{thm:commonedges} in the $\ell^1$ case.
\begin{proposition}
\label{prop:thminl1case}
Let $u,v \in \R^n$ and let $J=\Vert \cdot \Vert_1$. Let $I$ be the index set where $u$ and $v$ are nonzero, $I = \{i ~ | ~   u_i \neq 0, v_i\neq 0 \}$. Then $sign(u_i) = sign(v_i)$ for all $i \in I$ is equivalent to the existence of a  $p \in \partial J(u) \cap \partial J(v)$.
\end{proposition}

As we can see the signs of $u$ and $v$ on the common support have to coincide. For relating a common subgradient and finally an equal support with equal signs to the Bregman distance, it helps to write the Bregman distance as the following sum:
\begin{align}
\label{eq:ell1bregdist}
D_{\| \cdot \|_1}^p(v,u) = \sum_j (\text{sign}(v(j)) - p(j)) v(j).
\end{align}
The following lemma gives the conditions under which the support as well as the signs of nonzero entries of two elements $u$ and $v$ coincide:
\begin{lemma}
\label{lem:samesupport_samesign}
Let $u,v \in \R^n$ and let $J=\Vert \cdot \Vert_1$. Then 
\begin{equation}
	D_{\| \cdot \|_1}^p(v,u) =0 \qquad \forall~p \in \partial \| u \|_1. 
\end{equation}
implies that sign$_0(u) v \geq 0$, i.e. $v(i)$ cannot have opposite signs as $u(i)$.\\
Moreover sign$_0(u)$=sign$_0(v)$ if and only if 
\begin{equation}
	D_{\| \cdot \|_1}^p(v,u) =0, \quad D_{\| \cdot \|_1}^q(u,v)=0 \qquad \forall~p \in \partial \| u \|_1, q \in \partial \| v \|_1. 
\end{equation}
\end{lemma}
\vspace*{-12pt}
\begin{proof}
Let $I = \{j ~|~ v(j) \neq 0, ~ \text{sign}_0(u(j)) \neq \text{sign}_0(v(j)) \}$. Using \eqref{eq:ell1bregdist} with the subgradient $p = \text{sign}_0(u)$, it is easy to see that
$$ D_{\| \cdot \|_1}^p(v,u) = \sum_{j \in I } (\text{sign}_0(v(j)) - \text{sign}_0(u(j))) v(j). $$
Since the signs of $v(j)$ and $u(j)$ in each summand are different, any summand would be strictly positive. Due to $D_{\| \cdot \|_1}^p(v,u) =0$, $I$ has to be empty. Particularly, the support of $v$ is contained in the support of $u$. If additionally  $D_{\| \cdot \|_1}^q(u,v) =0$ for all $q$, the support of $u$ is contained in the support of $v$ and the $u(i)$ cannot have opposite signs as the $v(i)$, which leads to sign$_0(u)$=sign$_0(v)$. The reverse implication follows straight-forward using equation \eqref{eq:ell1bregdist}.
\end{proof}

Keeping the idea of regularizing with the Bregman distances between several elements in mind, the natural question arising from lemma \ref{lem:samesupport_samesign}, proposition \ref{prop:thminl1case}, and theorem \ref{thm:commonedges} is if a zero Bregman distance is too restrictive considering the `same sign' (respectively the `edges pointing in the same direction') conditions. 

In the context of recent literature (cf. \cite{Spr11} and the references therein), a zero Bregman distance with respect to the $\ell^1$ norm can be seen as a special case of collaborative sparsity, where not only the support, but also the signs should coincide. In case the equality of signs is too restrictive, the independence of signs can be achieved by considering the infimal convolution of Bregman distances as a regularization, as Lemma \ref{lem:samesupport} shows.

The infimal convolution of two convex functionals $J(\cdot )$ and $R(\cdot )$ is defined as follows
\begin{align}
J(v)\square R(v) = \inf_{\phi, \psi, v = \phi + \psi} J(\phi) + R(\psi).
\end{align}
In the following we are going to further investigate the infimal convolution of Bregman distances, not to be confused with the Bregman distance for the infimal convolution of functionals. More precisely, the infimal convolution is carried out for different subgradients (related to the second argument) considering the Bregman distance as a functional of the first argument, which is indeed convex as soon as $J$ is.

\begin{lemma}
\label{lem:l1infconvrepresentation}
The infimal convolution between the $\ell^1$ Bregman distance $D_{\| \cdot \|_1}^p(v,u) $ and $D_{\| \cdot \|_1}^{-p}(v,-u)$ with respect to $v$ is given by
\vspace*{-6pt}
\begin{equation}
\label{eq:l1infconvrepresenation}
	D_{\| \cdot \|_1}^p(v,u) \square D_{\| \cdot \|_1}^{-p}(v,-u) = \sum_{i=1}^N |v(i)|(1-|p(i)|).
\end{equation}
\vspace*{-6pt}
\end{lemma}
\begin{proof}
We compute
\begin{align}
D_{\| \cdot \|_1}^p(v,u) \square D_{\| \cdot \|_1}^{-p}(v,-u) &= \inf_{\phi, \psi, v = \phi + \psi} \|\phi\|_1 - \langle \phi, p \rangle + \|\psi\|_1 + \langle \psi, p \rangle, \nonumber \\
&= \inf_{\phi, \psi, v = \phi + \psi} \sum_i (|\phi(i)| - p(i) \phi(i) + |\psi(i)| + p(i) \psi(i)), \nonumber  \\
&= \inf_{\phi, \psi, v = \phi + \psi} \sum_i |\phi(i)|(1 -p(i) \text{sign}(\phi(i))) \nonumber \\
& \qquad \qquad \ +   \sum_i |\psi(i)|(1 +p(i) \text{sign}(\psi(i))),  \nonumber \\
&\geq\inf_{\phi, \psi, v = \phi + \psi} \sum_i |\phi(i)|(1 -|p(i)| ) +   \sum_i |\psi(i)|(1 - |p(i)|)\nonumber \\
&\geq \sum_i |v(i)|(1 -|p(i)|).  \nonumber 
\end{align}
On the other hand, $D_{\| \cdot \|_1}^p(v,u) \square D_{\| \cdot \|_1}^{-p}(v,-u) \leq \sum_i |v(i)|(1 -|p(i)|)$ holds for $\phi(i) = v(i)$ and $\psi(i) = 0$ if $\text{sign}(v(i)) = \text{sign}(p(i))$ and  $\phi(i) = 0$ and $\psi(i) = v(i)$ if $\text{sign}(v(i)) = -\text{sign}(p(i))$.
\end{proof}

As we can see $D_{\| \cdot \|_1}^p(v,u) \square D_{\| \cdot \|_1}^{-p}(v,-u)$ does not penalize $v(i)$ if $|p(i)|=1$ and thus is independent of the sign of $u(i)$. We can now formulate the property of two elements having the same support in terms of the infimal convolution of Bregman distances:
\begin{lemma}
\label{lem:samesupport}
Let $u,v \in \R^n$. Then
\begin{equation}
	D_{\| \cdot \|_1}^p(v,u) \square D_{\| \cdot \|_1}^{-p}(v,-u)  =0 \qquad \forall~p \in \partial \|u\|_1 \label{equalsupport1}
\end{equation}
if and only if the support of $v$ is contained in the support of $u$.
Moreover $u$ and $v$ have equal support if and only if \eqref{equalsupport1} and
\begin{equation}
	D_{\| \cdot \|_1}^q(u,v) \square D_{\| \cdot \|_1}^{-q}(u,-v)  =0 \qquad \forall~q \in \partial \|v\|_1
\end{equation}
hold.
\end{lemma}
\begin{proof}
Using the representation \eqref{eq:l1infconvrepresenation} as well as the fact that the sum is zero if and only if each summand is zero, the lemma simplifies to stating that the support of $v$ is contained in the support of $u$ if and only if 
\begin{equation}
\label{eq:componentwise}
 |v(i)|(1 -|p(i)|) = 0 \qquad \forall~p(i) \in \text{sign}(u(i)) 
\end{equation}
for all $i$. If there exists an $i$ for which $u(i)=0$ but $v(i)\neq 0$, $p(i)=0$ is a valid choice for the $i$-th component of the subgradient and equation \eqref{eq:componentwise} is violated. On the other hand, if there exists a $p(i)$ such that $p(i) \in \text{sign}(u(i))$ but $|v(i)|(1 -|p(i)|) \neq 0$, then  $|v(i)|>0$ and $|p(i)|<1$, which, however, means that $u(i) = 0$. 
\end{proof}

This motivates us to consider the infimal convolution of Bregman distances in the TV case as well. Our goal is to show that the infimal convolution $D_{J}^p(v,u) \square D_{J}^{-p}(v,-u)$ with respect to TV can be zero independent of the direction of the edge set. 

\subsection{Infimal Convolution of Bregman Distances and Total Variation}
As a first auxiliary result and as a reasoning for also considering the Bregman distance using $-p$, the following Lemma states how the subdifferential of a one-homogeneous functional behaves with respect to a change of sign.
\begin{lemma} 
 Let $J$ be convex and one-homogeneous. Then $\partial J(-u) = - \partial J(u)$. 
\end{lemma}
\begin{proof}
It is well known that the subdifferential of a one-homogeneous function can be characterized as follows:
$$ \partial J(u) = \{ p ~|~ J(u) = \langle p, u \rangle, ~ J(v) - \langle p,v\rangle \geq 0 \ \forall v \}. $$
Let $p \in \partial J(-u)$. By the above characterization, we have $J(-u) = \langle p, -u \rangle$. On the other hand $J(-u)=J(u)$ such that $J(u) = \langle -p, u \rangle$. Again by the above characterization, we know that $J(v) - \langle p,v\rangle \geq 0$ for all $v$. Choosing $-v$ instead of $v$ it also holds that $J(v) - \langle -p,v\rangle \geq 0$ for all $v$. The latter condition together with the previously shown fact that $J(u) = \langle -p, u \rangle$ shows $-p \in \partial J(u)$. Similar arguments show the reverse also holds.
\end{proof}

As a consequence it indeed makes sense to consider the infimal convolution of the Bregman distances with subgradients of opposite sign. An explicit formula seems out of reach currently, but we can provide upper and lower bounds, from which we can gain further insight. 

\begin{proposition}
\label{prop:infconvreformulation}
Let $v \in C^1$, $p=\nabla \cdot g \in \partial J(u)$ for $J=TV$. Then 
\begin{equation}
\label{eq:infconvTVrepresen}
\int_\Omega (1-|g|) |\nabla v|~dx\leq  D_{J}^p(v,u) \square D_{J}^{-p}(v,-u) \leq \int_\Omega |\nabla v| ~dx- | \int_\Omega \nabla v \cdot g~dx|.
\end{equation}
\end{proposition}
\begin{proof}
We use the definition of the infimal convolution as the infimum of the functional
$$ F(w) = J(v-w)+J(w) - \langle p, v-w\rangle + \langle p, w \rangle . $$
By a density argument we can compute the infimal convolution as the infimum over all $C^1$ functions $w$ of 
$$ F(w) = \int_\Omega  |\nabla (v-w)|+ |\nabla w| + g \cdot \nabla (v - 2w) . $$
With the Cauchy-Schwarz inequalities on the gradient fields we obtain 
$$ F(w) \geq  \int_\Omega (1-|g|)( |\nabla (v-w)|+ |\nabla w|)~dx, $$
and the triangle inequality yields the lower bound. 

Choosing $w=0$ and $w=v$ and using the regularity of $v$ to integrate by parts we see
$$ \inf_w F(w) \leq F(0) = \int_\Omega |\nabla v| + g \cdot \nabla v ~dx $$
and 
$$ \inf_w F(w) \leq F(v) = \int_\Omega |\nabla v| - g \cdot \nabla v ~dx, $$ 
from which we obtain the upper bound.
\end{proof}

With the infimal convolution of Bregman distances we obtain a more general version of the morphological invariance, the infimal convolution also vanishes for nonmonotone transformations of the intensities:
\begin{lemma}
Let $v=f(u)$ for some function $f: \R \rightarrow \R  \in BV$ and $J=TV$. Then 
\begin{equation}
	D_J^p(v,u) \square D_J^{-p}(v,-u)= 0 \qquad \forall~p \in \partial J(u). 
\end{equation}
\end{lemma}
\vspace*{-12pt}
\begin{proof}
According to the Jordan decomposition (cf. \cite{Bia11} for a nice extension to higher dimensions), any function $f: \R \rightarrow \R \in BV$ can be written as $ f = f_1 - f_2 $ for monotonically increasing functions $f_1$ and $f_2$. Now lemma \ref{lem:monotoneBregDistZero} tells us that $J(f_1(u))  - \langle p, f_1(u) \rangle = 0$ and $J(f_2(u))  - \langle p, f_2(u) \rangle = 0$ for all $p$, which immediately yields $D_J^p(v,u) \square D_J^{-p}(v,-u)= 0 $  for all $p$. 
\end{proof}

Another generalization concerns the case of piecewise constants, now we obtain a version independent of the signs of jumps as desired:
\begin{theorem}
Let $u$ and $v$ be piecewise constant with $C^1$ discontinuity sets
\begin{align}
E(u) &= \Gamma \cup \Sigma, \\
E(v) &= \Gamma \cup \Upsilon, 
\end{align}
such that $\Gamma$, $\Sigma$, and $\Upsilon$ have pairwise positive distance.  
Then there exists $p \in \partial J(u)$ such that
$$ D_{J}^p(v,u) \square D_{J}^{-p}(v,-u) = 0 .$$
\end{theorem}
\vspace*{-12pt}
\begin{proof}
Due to the definition of the infimal convolution and its nonnegativity it suffices to find $w \in BV(\Omega)$ and
$p = \nabla \cdot g \in \partial J(u)$ such that
$$ F(w) = J(v-w)+J(w) - \langle p, v-w\rangle + \langle p, w \rangle  $$
vanishes. 
For this sake we construct $w$ with jump set $\Gamma \cup \Upsilon \cup \Theta$ such that $\Theta$ has positive distance from $\Gamma$, $\Upsilon$, and $\Sigma$, such that
$w=0$ in a neighbourhood of $\Upsilon$ and of the part of $\Gamma$ where $u$ and $v$ have the same jump sign, and $w=v$ in a neighbourhood of the part of $\Gamma$ where $u$ and $v$ have the opposite jump sign.
It is well-known that for a piecewise constant $u$ the vector field $g$ needs to satisfy
$g\cdot n=$sign$([u])$ on $\Gamma \cup \Sigma$, but can be chosen arbitrarily (up to the constraint on the supremum norm) on $\Upsilon \cup \Theta$ to obtain a subgradient $p=\nabla \cdot g \in \partial J(u)$. Hence, let 
$$ g\cdot n = - \text{ sign}([w]) \qquad \text{on } \Theta \quad \text{ and } \quad g\cdot n = \text{ sign}([v]) \qquad \text{on } \Upsilon.$$

With the above choices, we find
$$ F(w) = \int_{\Gamma \cup \Upsilon} (|[v-w]| + |[w]| - g\cdot n [v-2w])~d\sigma 
+ 2 \int_\Theta ( |[w]|+ g \cdot n [w]).$$
The second integral vanishes due to the above properties of $g \cdot n$. To show that the first integral vanishes we can inspect the integrand on the different parts. On $\Upsilon$ the
integrand equals $|[v]|-g\cdot n [v]$, which vanishes because of the above choice of $g$. 
On the part of $\Gamma$ where $u$ and $v$ have opposite jump sign the same applies. On the part of $\Gamma$ with opposite jump sign the integrand equals 
$$|[v]|+g\cdot n [v]= |[v]| - |[v]| = 0. $$
Hence, we conclude $F(w) = 0$.
\end{proof}

\section{Color Bregman Iterations}
\label{sec:colBreg}

Our above analysis shows that in the case of TV regularization, having the same edge set and direction of the edges means having a common subgradient. Let us consider the case of reconstructing signals $u_i$ from noisy measurements $f_i$. Considering a regularization which encourages multiple channels to share the edge sets and their direction could therefore lead to a model of the form
\begin{equation}
	\sum_{i=1}^M \left( \frac{1}2 \Vert K_i u_i -f_i \Vert^2 + \alpha (J(u_i) - \langle p_0, u_i \rangle ) \right)\rightarrow \min_{u \in {\cal U}^M, ~ p_0 \in \cap_i \partial J(u_i)} 
\end{equation}
where $K_i$ denote linear operators which are specific to the measurements or the reconstruction task and $\alpha$ is a weight for the regularization term. Unfortunately, the constraint of a common subgradient leads to a highly degenerate set of constraints and even the existence of a minimizer for the above problem is unclear. Additionally, the above problem is nonconvex, such that we suggest to instead use an iterative procedure where each iteration is the solution of a convex minimization problem using the Bregman distances to all $u_i$ of the previous iteration as a regularization. 

\subsection{Iteration Schemes}
Let $K: {\cal U}^M \rightarrow {\cal Y}^M$ be a block diagonal operator given by $u \mapsto (K_i u_i)$. We propose to find the iterate $u^{k+1}$ as the vector of minimizers 
\begin{equation} \label{eq:primalupdate}
	u_i^{k+1} \in \text{arg}\min_{u_i \in {\cal U}} \left( \frac{1}2 \Vert K_i u_i  -f_i \Vert^2 + \alpha \sum_{j=1}^M w_{i,j} D_J^{p_j^k}(u_i,u_j^k) \right)
\end{equation}
for nonnegative weights $w_{i,j}$ with $ \sum_j w_{i,j} =1$. The optimality condition to the minimization problems \eqref{eq:primalupdate} can be written in vector form as 
\begin{equation}
	p^{k+1} = W p^k - \lambda K^* (K u^{k+1} - f), \qquad p_i^{k+1} \in \partial J(u_i^{k+1}),
	\label{eq:dualupdate}
\end{equation}
where we used $\lambda = \frac{1}\alpha$ and the entries of the matrix $W$ are $w_{i,j}$. Therefore, the rows of the weight matrix $W$ sum to one
\begin{equation}
	\sum_{j=1}^M w_{i,j} = 1 \qquad \forall~i=1,\ldots,M, \label{eq:Wnormalization}
\end{equation}
such that $(1,1,\ldots,1)^T$ is an eigenvector for $W$ with eigenvalue $1$.

The alternative scheme using infimal convolution of Bregman distances is given by
\begin{align} \label{eq:primalupdateinfconv}
	u_i^{k+1} \in \text{arg}\min_{u_i \in {\cal U}} \bigg( \frac{1}2 \Vert K_i u_i  -f_i \Vert^2 &+ \alpha w_{i,i} D_J^{p_i^k}(u_i,u_i^k)  \nonumber \\
	&+ \alpha \sum_{j=1,j\neq i}^M w_{i,j} ( D_J^{p_j^k}(u_i,u_j^k)
	\square D_J^{-p_j^k}(u_i,-u_j^k))  \bigg).
\end{align}
Note that we treat the diagonal term differently, namely without infimal convolution. This is natural however, since the different sign of jumps should only be relevant for different channels, while we want to use the information on the $i$-th channel itself as in the original Bregman iteration. 
Optimality conditions can be stated as above, and one observes that it is also important to use $D_J^{p_i^k}$ only in the diagonal term in order to really obtain an update equation for the subgradients.

Since the subgradients of the infimal-convolution are not very intuitive, we derive an alternative formulation using new auxiliary variables as in the definition of the infimal convolution
\begin{align} 
	u_i^{k+1} \in \text{arg}\min_{u_i \in {\cal U}} ~\min_{z_{ij} \in {\cal U}} 
	{\Big (} & \frac{1}2 \Vert K_i u_i  -f_i \Vert^2 + \alpha w_{i,i} D_J^{p_i^k}(u_i,u_i^k) + \nonumber \\ & \alpha \sum_{j=1,j\neq i}^M w_{i,j} ( J(u_i-z_{ij}) + J(z_{ij}) - \langle p_j^k, u_i- z_{ij} \rangle + \langle p_j^k, z_{ij} \rangle)  { \Big )}.
\end{align}
Now we can easily derive optimality with respect to $u_i$ and $z_{ij}$ 
\begin{align}
K_i^*(K_i u_i^{k+1} - f_i) + \alpha w_{i,i} ( p_i^{k+1} - p_i^k ) + \alpha \sum_{j \neq i} w_{i,j} (r_{ij}^{k+1} - p_j^k) &= 0, \label{infconvopt1}\\
2p_j^k - r_{ij}^{k+1} + g_{ij}^{k+1} &= 0, \label{infconvopt2} \\
p_i^{k+1} \in \partial J(u_i^{k+1}), ~ 
g_{ij}^{k+1} \in \partial J(z_{ij}^{k+1}), ~ 
r_{ij}^{k+1} \in \partial J(u_i^{k+1}-z_{ij}^{k+1}). & \label{infconvopt3}
\end{align}

\subsection{Stationary Solutions} The following subsection investigates possible stationary points of the two iterative schemes described above. Additionally to the results presented below, we include an explicit example for the behavior of color Bregman iteration in the case where $f_i = c_i K \nu$ for $\nu$ being a singular vector in the supplementary material to this manuscript. We refer the reader to \cite{benningBurger} for details regarding the concept of singular values for nonlinear variational methods.

\subsubsection{Color Bregman Iteration}
In the following we study potential stationary solutions of the iteration scheme, i.e. $(u,p) \in {\cal U}^M \times ({\cal U}^*)^M$ such that
\begin{equation}
	p = Wp - \lambda K^* (K u - f), \qquad p_i \in \partial J(u_i).
	\label{eq:stationary}
\end{equation}

Of particular interest are of course solutions with zero residual $K^*(Ku-f) = 0$, for which one needs $p=Wp$. This implies $p = p_1 e_1$, where $e_1$ is an eigenvector of $W$ with eigenvalue $1$. We are interested in the case of $e_1=(1,1,\ldots,1)^T$ being the unique positive normalized eigenvector, which is guaranteed e.g. by the Perron-Frobenius theorem if all $w_{ij}$ are positive. Then the above property would imply $p_i=p_1$ for all $i$ and we would have found a solution for which all channels share a common subgradient. 

In the general case there might however be different stationary solutions. If $W$ is symmetric, which appears to be the natural choice, then we have that $e_1^T$ is a right eigenvector of $W$, which implies
$$ e_1^T p = e_1^T W p - \lambda e_1^T K^*(Ku-f) = e_1^T p - \lambda e_1^T K^*(Ku-f) , $$
i.e. $ e_1^T K^*(Ku-f) = 0$. This means that the sum of the residuals is zero for any stationary solution, but not necessarily each residual itself. 	 

For $W$ symmetric being such that $WK^* = K^* W$ one can derive more detailed arguments. Note that the latter property is satisfied automatically in the prominent case of all operators $K_i$ being identical. In this case we obtain that for a stationary solution $(I-W)p$ is in the range of $K^*$, i.e. there exists $q \in {\cal Y}$ such that
$$ \alpha (I-W)q =  (f - Ku) , \qquad K_i^* q_i \in \partial J(u_i). $$
This relation is a saddle point condition for the functional
\begin{equation}
	{\cal S}(u,q) = \sum_{i=1}^M J(u_i) - \langle q, K u -f \rangle - \frac{\alpha}2 \langle (I-W) q, q \rangle. 
\end{equation}
Note that ${\cal S}$ is convex with respect to $u$, hence a natural condition would be concavity with respect to $q$. The latter is satisfied since under the above conditions (symmetry, nonnegative entries, row sums equal to one) the matrix $I-W$ is positive semidefinite.
Note that $e_1$ is always in the nullspace of $I-W$, hence it will not be positive definite.
This property of $W$, together with the saddle-point interpretation of the stationary solution, will also be crucial in the convergence analysis below. Notice that in case that there exists a solution of $Ku=f$ with $q = Wq$ and $p_i = K_i^* q_i \in \partial J(u_i)$, then $(u,q)$ is a saddle point of ${\cal S}$ (independent of $\alpha$). 
\subsubsection{Infimal Convolution Bregman Iteration}
Let us now look at possible stationary solutions of the iteration using infimal convolutions of Bregman distances, i.e. at solutions of the equations
\begin{align}
K_i^*(K_i u_i - f_i)+ \alpha \sum_{j \neq i} w_{i,j} (r_{ij}- p_j) &= 0, \label{infconvstat1}\\
2p_j - r_{ij} + g_{ij} &= 0, \label{infconvstat2} \\
p_i \in \partial J(u_i), ~ 
g_{ij} \in \partial J(z_{ij}), ~ 
r_{ij} \in \partial& J(u_i-z_{ij}).  \label{infconvstat3}
\end{align}
 Again we are particularly interested in the case of a zero residual  $K^*(K u - f)=0$. We can state
\begin{proposition}
A necessary and sufficient condition for a solution to \eqref{infconvstat1}, \eqref{infconvstat2} and \eqref{infconvstat3} to meet $K^*(K u - f)=0$ is that there exist $z_{ij}$ which meet
$$
D_J^{p_j}(u_i - z_{ij}, u_j) = 0, \qquad
D_J^{-p_j}( z_{ij}, -u_j) = 0. 
$$
\end{proposition}
\vspace*{-12pt}
\begin{proof}
The above condition is sufficient since it allows us to choose $r_{ij} = p_j \in \partial J(z_{ij})$ and $g_{ij}=-p_j  \in \partial J(u_i - z_{ij})$, such that \eqref{infconvstat1} immediately yields $K_i^*(K_i u_i - f_i)=0$ for all $i$. To see that the above condition is necessary, we consider the dual product 
\begin{align*}
\frac{1}{\alpha} \langle K_i^*(f_i - K_iu_i) , u_i \rangle &=  \langle \sum_{j \neq i} w_{i,j} (r_{ij}- p_j), u_i \rangle =  \sum_{j \neq i} w_{i,j} \langle  r_{ij}- p_j, u_i - z_{ij} + z_{ij}\rangle, \\
 &=  \sum_{j \neq i} w_{i,j} \bigg( \underbrace{\langle  r_{ij}, u_i - z_{ij} \rangle - \langle p_j, u_i - z_{ij} \rangle}_{= D_J^{p_j}(u_i-z_{ij}, u_j)}  + \langle  \underbrace{r_{ij}- p_j}_{=g_{ij} + p_j}, z_{ij}\rangle  \bigg), \\
 &=  \sum_{j \neq i} w_{i,j} \bigg( D_J^{p_j}(u_i-z_{ij}, u_j)  + \underbrace{\langle  g_{ij} , z_{ij}\rangle - \langle -p_j, z_{ij} \rangle}_{=D_J^{-p_j}(z_{ij}, -u_j)}  \bigg).
\end{align*}
Obviously, the above being equal to zero is necessary for $K^*(K u - f)=0$. Due to the nonnegativity of the Bregman distance, the above sum can only be equal to zero if each summand is zero, which yields the assertion.
\end{proof}

As we can see, any stationary point with zero residual of the color Bregman iteration also is a stationary point with zero residual of the infimal convolution Bregman iteration. Additionally, we gain all points that can be decomposed into a part with a zero Bregman distance to $u_j$ and a part with zero Bregman distance to $-u_j$. Particularly, we can see the independence of the sign of the $p_j$. By choosing $z_{ij}=u_i$ or $z_{ij}=0$, we see that any point with zero residual and a subgradient $p$ such that for each $i$ either $p \in \partial J(u_i)$ or $-p \in \partial J(u_i)$ holds, is a stationary point of the infimal convolution Bregman iteration.

\subsection{Primal-Dual and Dual Formulation}
\label{sec:primaldual}
Under the condition introduced in the last section, i.e. $K^* W = W K^*$, we can employ the saddle-point structure to provide alternative formulations of the iterative scheme \eqref{eq:dualupdate}. Such a reformulation is also well-known in the case of the original Bregman iteration and can be interpreted as an Augmented Lagrangian method.   For initial values of $p$ in the range of $K^*$ it is straight-forward to see that $p^k = K^* q^k$ and we can rephrase \eqref{eq:dualupdate} as
\begin{equation}
	q^{k+1} = W q^k - \lambda (K u^{k+1} - f), \qquad K_i^* q_i^{k+1} \in \partial J(u_i^{k+1}).
	 \label{eq:primaldual}
\end{equation}
With ${\cal S}$ defined as above, these conditions can be reinterpreted as
\begin{align}
  	u^{k+1} &\in \text{arg}\min_u  {\cal S}(u,q^{k+1}), \\
	q^{k+1} &= q^k + \partial_q {\cal S}(u^{k+1},q^k). \label{qupdate}
\end{align}
Hence, the scheme is minimizing with respect to $u$ at fixed $q$, while $q^{k+1}$ is obtained with an ascent step with respect to $q$. 

We observe that the primal-dual iteration \eqref{eq:primaldual} can be considered as an alternative in the case $K^* W \neq W K^*$. Eliminating $q^{k+1}$ it can be rephrased as
$$  K_i^* (W q^k)_i - \lambda K_i^*(K u^{k+1} - f)_i \in \partial J(u_i^{k+1}), $$
which is the optimality condition for the variational problem
\begin{equation}
	u^{k+1} \in \text{arg}\min_{u \in {\cal U}^M} \left( \frac{\lambda}2 \norm{K u -f}^2 + \sum_i J(u_i) - \langle W q^k, K u - f \rangle \right).  \label{eq:alternativeiteration}
\end{equation}
 
Finally we shall also state a dual form. We will give a purely formal derivation here and not justify the duality, since we will use the dual formulation as a motivation of the convergence analysis below only, where we are interested in the convergence of $u^k$ as well. Under appropriate regularity $p_i \in \partial J(u_i)$ can be reformulated to $u_i \in \partial J^*(p_i)$ with $J^*$ being the convex conjugate of $J$. Note that $u_i^{k+1} \in \partial J^*(K_i^* q_i^{k+1})$ implies 
$K u_i^{k+1} \in \partial_q J^*(K_i^* q_i^{k+1})$. Hence, the iteration 
$$ q^{k+1} + K u^{k+1} - f =  q^k - \alpha (I-W) q^k  $$ 
can also be interpreted as a forward-backward splitting method on the dual functional
\begin{equation}
	D(q) = \frac{\alpha}2 \langle (I-W) q,q \rangle - \langle f, q \rangle + \sum_i J^*(K_i^* q_i),
\end{equation}
where the forward splitting is applied to the first term and the backward splitting to the other two.

\subsection{Well-Definedness}

We assume that ${\cal U}$ is a Banach space being the dual of some Banach space ${\cal V}$ and that
${\cal R}(K_i^*) \subset {\cal V}$ for all $i=1,\ldots,M$. This is satisfied e.g. if $K_i$ is the adjoint of some operator $L_i$ mapping to ${\cal V}$. 

\begin{lemma}
\label{lem:minimizerExists}
Let the sublevel sets of $u \mapsto \Vert K_i u \Vert ^2 + J(u)$ be bounded in ${\cal U}$ and let 
$H: {\cal U} \rightarrow \R \cup \{+\infty\}$ be a proper nonnegative convex functional. Then for each $q \in {\cal Y}$ and each $\epsilon > 0$ there exists a solution of the variational problem
\begin{equation}
\label{eq:minimizerExists}
\frac{1}2 \Vert K_i u_i - f_i \Vert^2 + \epsilon J(u_i) - \langle K_i^* q, u_i \rangle + H(u_i) \rightarrow \min_{u_i \in {\cal U}} . 
\end{equation}
Moreover, the solution is unique if $K_i$ is injective. 
\end{lemma}
\begin{proof}
By adding a constant independent of $u_i$ we can rewrite the minimization equivalently as 
\begin{equation}
\frac{1}2 \Vert K_i u_i - f_i - q \Vert^2 + \epsilon J(u_i) + H(u_i) \rightarrow \min_{u_i \in {\cal U}} . 
\end{equation}
This functional is convex and hence weak-* lower semicontinuous on bounded sets. Moreover, the nonnegativity of $H$ implies an upper bound on $\Vert K_i u_i \Vert ^2 + J(u)$ on each sublevel set of the functional. Thus, the latter are bounded and thus weak-* compact by the Banach-Alaoglu theorem. These two properties yield the existence of a minimizer by a standard argument. 
\end{proof} 

\begin{theorem}
Let the sublevel sets of $u \mapsto \Vert K_iu \Vert ^2 + J(u)$ be bounded in ${\cal U}$, let 
$p_i^0 =K_i^* q_i^0  \in \partial J(u_i^0)$. Then, the sequence of iterates $(u_i^k,p_i^k)$ defined by
\eqref{eq:primalupdate} is well-defined and satisfies \eqref{eq:dualupdate}.
\end{theorem}
\begin{proof}
Let $p_i^{k} =K_i^* q_i^{k}  \in \partial J(u_i^{k})$ be met (holds for $k=0$ by the conditions of the theorem). We determine the next iterates  $(u_i^{k+1,}p_i^{k+1})$ by \eqref{eq:primalupdate}. By defining $H(u_i) = \alpha \sum_{j\neq i} w_{i,j} D_J^{p_j^k}(u_i,u_j^k)$, which is proper, convex and nonnegative), we see that each minimization problem for the $u_i$ has the form of \eqref{eq:minimizerExists} with $\epsilon = \alpha$ such that lemma \ref{lem:minimizerExists} guarantees the existence of a solution. The optimality conditions yields
\begin{align}
K_i^*(K_i u_i^{k+1} - f_i) + \alpha p_i^{k+1} - \alpha \sum_j w_{i,j}  p_j^{k} = 0
\end{align} 
for $p_i^{k+1} \in \partial J(u_i^{k+1})$, such that \eqref{eq:dualupdate} holds for $\lambda = 1/ \alpha$. Induction yields the result for all $k$. 
\end{proof}

In a similar way the Bregman iteration with infimal convolutions can be analyzed, noticing that
by the properties of the infimal convolutions the variational problem in each iteration can still be 
cast in the form of Lemma \ref{lem:minimizerExists} if $\epsilon:=\alpha w_{i,i} > 0$.
\begin{theorem}
Let the sublevel sets of $u \mapsto \Vert K_iu \Vert ^2 + J(u)$ be bounded in ${\cal U}$, let 
$p_i^0 =K_i^* q_i^0  \in \partial J(u_i^0)$, and $w_{i,i} > 0$ for all $i$. Then, the sequence of iterates $(u_i^k,p_i^k)$ defined by
\eqref{eq:primalupdateinfconv} is well-defined.
\end{theorem}

\section{Convergence Analysis}
\label{sec:convergenceanalysis}

In the following we discuss the convergence analysis of the iteration scheme (setting $\lambda = 1$ for simpler notation)
\begin{equation}
	q^{k+1} = W q^k + f - K u^{k+1}, \qquad K_i^* q_i^{k+1} \in \partial J(u_i^{k+1}).
	\label{eq:finaliteration}
\end{equation}
Throughout the whole section we shall assume that
\begin{itemize}
	
\item $W$ is symmetric with nonnegative entries and row sums equal to one.

\item $W$ has a simple maximal eigenvalue $\lambda_1=1$.

\item The initial values satisfy $p_i^0=K^*q_i^0 \in \partial J(u_i^0)$.	

\end{itemize}
\ \\
We shall denote by $e_1=\frac{1}{\sqrt{M}}(1,1,\ldots,1)^T \in \R^M$ the normalized positive eigenvector corresponding to $\lambda_1$ and by $Q_1 \in \R^{M\times M}$ the projection matrix on the orthogonal space to $e_1$. 

\subsection{Auxiliary Estimates}

We start with some estimates on quantities potentially dissipated by the iteration. Those are norms of $q^k$ as well as Bregman distances between $u^k$ and some potential final state $\overline{u}$.

\begin{proposition} \label{dissprop1}
Let $(u^k,q^k)$ be a sequence generated by \eqref{eq:finaliteration}. Then the estimate
\begin{align}
&	\Vert Q_1 (q^{k+1} - q^k) \Vert^2 + \Vert e_1^T (q^{k+1} - q^k) \Vert^2  + 2D^{k,k+1}_{\text{symm}} + \norm{K(u^{k+1}-u^k)}^2 \nonumber \\ & \qquad \leq \lambda_2^2 \norm{Q_1(q^k-q^{k-1})}^2 + \Vert e_1^T (q^{k} - q^{k-1}) \Vert^2
	\label{eq:dissipation1} 
\end{align}
holds,
where $D^{k,k+1}_{\text{symm}}$ denotes the symmetric Bregman distance
\begin{equation}
	D^{k,k+1}_{\text{symm}} = \langle p^{k+1}- p^k, u^{k+1} - u^k \rangle,
\end{equation}
and $\lambda_2$ is the second largest eigenvalue of W.
\end{proposition} 
\begin{proof}
Subtracting \eqref{eq:finaliteration} for index $k$ and $k-1$ we obtain
$$ q^{k+1} - q^k  + K(u^{k+1}-u^k) = W(q^k-q^{k-1}) . $$
Taking the squared norm in the Hilbert space ${\cal Y}$ and expanding the square on the 
left-hand side yields 
\begin{equation}
\label{eq:qnorms}
 \norm{q^{k+1} - q^k}^2  + 2D^{k,k+1}_{\text{symm}} + \norm{K(u^{k+1}-u^k)}^2 = \norm{W(q^k-q^{k-1})}^2 . 
 \end{equation}
Now we can further expand the squared norms in $q$ into the parts in the subspaces spanned by $e_1$ and orthogonal to $e_1$, respectively. Finally, the standard spectral estimate
$$ \norm{Q_1 W(q^k-q^{k-1})} \leq \lambda_2 \norm{Q_1(q^k-q^{k-1})} $$ 
yields \eqref{eq:dissipation1}.
\end{proof}

As a direct consequence of \eqref{eq:dissipation1} we conclude after summation
\begin{equation}
	\Vert q^{k+1} - q^k  \Vert^2 + \sum_{j=1}^k \norm{K(u^{j+1}-u^j)}^2 + (1-\lambda_2^2) \sum_{j=1}^k \norm{Q_1(q^j-q^{j-1})}^2 
\leq \norm{q^1-q^0}^2 .
\end{equation} 
From the finiteness of the second sum and equation (\ref{eq:finaliteration}) we can immediately conclude
\begin{equation}
 Q_1((W-I) q^k + f - Ku^{k+1} )\rightarrow 0 	. 
\end{equation}
Let us define the residual at the $k$-th iteration as
\begin{equation}
\label{eq:residualdef}
r^k = (W-I)q^k + f - Ku^k.
\end{equation}
We can state:
\begin{corollary}
\label{cor:monoton1}
The quantities $\|r^k\|$ and $\|q^{k+1} - q^k\|$ are monotonically decreasing.
\end{corollary}
\begin{proof}
With \eqref{eq:qnorms} and $\Vert W \Vert =1$, which follows from the above properties, we have
$$ \|W(q^{k+1} - q^k)\| \leq \|q^{k+1} - q^k\| \leq  \|W(q^{k} - q^{k-1})\| \leq 
 \|q^{k} - q^{k-1}\|. $$
The decrease of $\| r^k\|$ follows from
\begin{align*}
r^{k} &= (I-W) q^{k} + (f - Ku^{k}) \\ 
    &= \underbrace{ q^{k} - Wq^{k-1} + (f - Ku^{k})}_{=0} - W(q^{k} - q^{k-1}).
\end{align*} 
\end{proof}

Another auxiliary estimate is the following:
\begin{proposition}  \label{dissprop2}
Let $(u^k,q^k)$ be a sequence generated by \eqref{eq:finaliteration} and let $\overline{u} \in {\cal U}^M$, $\overline{q} \in {\cal Y}^M$. Then the estimates
\begin{equation}
 \norm{q^{k+1}-\overline{q}}^2 + 2 \langle p^{k+1} - K^* \overline{q}, u^{k+1} - \overline{u} \rangle \leq \norm{W(q^k - \overline{q})}^2 + 2  \langle q^{k+1}- \overline{q}, r \rangle 
 \label{eq:dissipation2}
\end{equation}
and 
\begin{equation}
 \norm{q^{k+1}-\overline{q}}^2 + 2 \langle p^{k+1} - K^* \overline{q}, u^{k+1} - \overline{u} \rangle
 + \norm{Ku^{k+1}-K\overline{u}}^2 \leq \| W(q^k - \overline{q}) \|^2 + 2\langle W(q^k - \overline{q}), r\rangle + \|r\|^2  
 \label{eq:dissipation2a}
\end{equation}
hold, where $r$ is the residual
\begin{equation}
	r =  (W-I)\overline{q} + f- K\overline{u}. \label{eq:residual}
\end{equation}
\end{proposition}
\vspace*{-12pt}
\begin{proof}
Adding and subtracting $\overline{q}$ and $K\overline{u}$ as well as some rearrangement in \eqref{eq:finaliteration}, yields
\begin{equation} 
\label{eq:optiWithQ}
 q^{k+1} - \overline{q} + K(u^{k+1} - \overline{u}) = W(q^k - \overline{q}) + r. 
 \end{equation}
Taking now a squared norm in ${\cal Y}$ and expanding squares on both sides, yields
\eqref{eq:dissipation2a}.
 Taking the scalar product with $q^{k+1} -\overline{q}$ 
and an application of Young's inequality, yields \eqref{eq:dissipation2}.
\end{proof} 

The above result is particularly interesting if
 $K_i^* \overline{q}_i \in \partial J(\overline{u_i})$, since then 
the second term on the left-hand side of \eqref{eq:dissipation2} and \eqref{eq:dissipation2a} is a  symmetric Bregman distance between $u^{k+1}$ and $\overline{u}$ and in particular nonnegative.

A straight-forward computation further shows:
\begin{proposition}  \label{dissprop3}
Let $(u^k,q^k)$ with $p^k=K^* q^k$ be a sequence generated by \eqref{eq:finaliteration} and let $\overline{u} \in {\cal U}^M$, $\overline{q} \in {\cal Y}^M$. Then the identity
\begin{align}
 D_J^{p^{k+1}}(\overline{u},u^{k+1})-D_J^{p^{k}}(\overline{u},u^{k}) 
= -&  D_J^{p^k}(u^{k+1},u^k)  - \norm{q^{k+1}-q^k}^2 \nonumber \\
									&+ \langle q^{k+1}-q^k, r - (I-W)(q^k - \overline{q}) \rangle
 \label{eq:dissipation3}
\end{align}
holds, where $r$ is the residual defined by \eqref{eq:residual}.
\end{proposition}
\begin{proof}
We rearrange equation \eqref{eq:optiWithQ} to 
\begin{align*}
&  q^{k+1} - q^k + q^k - \overline{q} + K(u^{k+1} - \overline{u}) = W(q^k - \overline{q}) + r. \\
\Rightarrow ~ &  K(u^{k+1} - \overline{u}) = -(I-W)(q^k - \overline{q}) + r - (q^{k+1} - q^k). 
  \end{align*}
  We take the dual product with $q^{k+1}-q^k$ to obtain
$$
 \langle K(u^{k+1} - \overline{u}), q^{k+1}-q^k \rangle = \langle r - (I-W)(q^k - \overline{q}), q^{k+1}-q^k \rangle - \| q^{k+1} - q^k\|^2. 
$$
The left hand side can be rewritten as
\begin{align}
 \langle K(u^{k+1} - \overline{u}), q^{k+1}-q^k \rangle &=  \langle u^{k+1} - \overline{u}, p^{k+1}-p^k \rangle, \nonumber  \\
 &= - \langle p^k, u^{k+1} \rangle  + J(\overline{u}) - \langle p^{k+1}, \overline{u}-u^{k+1} \rangle- (J(\overline{u}) - \langle p^{k}, \overline{u} \rangle), \nonumber \\
  &= J(u^{k+1})- \langle p^k, u^{k+1} \rangle  + D_J^{p^{k+1}}(\overline{u},u^{k+1})- (J(\overline{u})- \langle p^{k}, \overline{u} \rangle), \nonumber \\ 
  &=D_J^{p^k}(u^{k+1},u^k)   + D_J^{p^{k+1}}(\overline{u},u^{k+1})-D_J^{p^{k}}(\overline{u},u^{k})\nonumber  ,
\end{align}
which proves the proposition. 
\end{proof}

\subsection{Clean Data}
In the following we consider the case of clean data, i.e. we assume there exists $\overline{u} \in {\cal U}^M$, $\overline{q} \in {\cal Y}^M$ such that
\begin{equation}
	0 = r = (W-I) \overline{q} + f - K \overline{u}, \qquad K_i^* \overline{q}_i \in \partial J(\overline{u}_i). \label{eq:cleandata}
\end{equation}
Note that \eqref{eq:cleandata} is always satisfied in our original setting, i.e. if we have a solution of $Ku =f$ with joint subgradient, if $K_i=K_1$ for all $i$, and the source-type condition $\overline{p}_1 = K_1^* \overline{q}_1$. In this case we automatically have 
$$ (W-I) \overline{q} = (W-I) e_1 \overline{p}_1 = 0. $$
Using the fact that the terms depending on $r$ drop out in the auxiliary estimates, we can state the monotonic behavior of $q^k$ approaching $\overline{q}$: 
\begin{corollary}
\label{cor:monoton2}
The quantities $\|q^k -  \overline{q}\|$ and $\|W(q^k -  \overline{q})\|$ are monotonically decreasing.
\end{corollary}
\begin{proof}
Equation \eqref{eq:dissipation2} with $r=0$ and $\langle p^{k+1} - K^* \overline{q}, u^{k+1} - \overline{u} \rangle = D_{symm}(u^{k+1},\overline{u}) \geq 0 $ leads to $\|q^{k+1} -  \overline{q}\| \leq \|W(q^k -  \overline{q})\|$, which together with $\|W\| = 1$ yields the assertion. 
\end{proof}

Using $r=0$ in the auxiliary estimates from the previous section, we can obtain a convergence result for the iterative scheme:
\begin{theorem} \label{cleandataconvergence}
Let $r=0$ and let $(u^k,q^k)$ with $p^k=K^* q^k$ be a sequence generated by \eqref{eq:finaliteration}. 
Then there exists a subsequence $(u^{k_\ell},q^{k_\ell})$ converging in the weak-star respectively weak topology. The limit $(u,q)$ of each such subsequence satisfies
\begin{equation}
	0 = (W-I)q + f- Ku. 
\end{equation}
Moreover, $Q_1 q^k \rightarrow Q_1 \overline{q}$, and if $K_i^*$ is compact then
$	K_i^*q_i \in \partial J(u_i).$
\end{theorem} 
\begin{proof}
Summation of \eqref{eq:dissipation1} yields
\begin{equation}
	\norm{q^{k+1}-\overline{q}}^2 + 2 \sum_{j=1}^{k+1} \langle p^{j} - K^* \overline{q}, u^{j} - \overline{u} \rangle + (1-\lambda_2^2) \sum_{j=1}^k \norm{Q_1(q^j - \overline{q})}^2  \leq \norm{q^0 - \overline{q}}^2.
\end{equation}
This implies $Q_1 q^k \rightarrow Q_1 \overline{q}$.

From \eqref{eq:dissipation3} we deduce by applying Young's inequality
$$ 2 D_J^{p^{k+1}}(\overline{u},u^{k+1})- 2 D_J^{p^{k}}(\overline{u},u^{k}) 
 \leq - 2 D_J^{p^k}(u^{k+1},u^k)  - \norm{q^{k+1}-q^k}^2 + \norm{(I-W)(q^k - \overline{q})}^2. $$
Adding this estimate and \eqref{eq:dissipation2} we further deduce
\begin{align*}
 & 2 D_J^{p^{k+1}}(\overline{u},u^{k+1}) + 2 D_J^{p^k}(u^{k+1},u^k)  + \norm{q^{k+1}-\overline{q}}^2 + \\
 & 2 \langle p^{k+1} - K^* \overline{q}, u^{k+1} - \overline{u} \rangle + \norm{q^{k+1}-q^k}^2  \\
 & \qquad \leq 2 D_J^{p^{k}}(\overline{u},u^{k}) + \norm{W(q^k - \overline{q})}^2 + \norm{(I-W)(q^k - \overline{q})}^2.
\end{align*}
Since $I-W$ is positive semidefinite it is straightforward to see
$$ \norm{Wq}^2+\norm{(I-W)q}^2 \leq \norm{q}^2. $$
Using this on the right-hand side of the previous estimate together with summation yields
\begin{align*}
 & 2 D_J^{p^{k+1}}(\overline{u},u^{k+1}) + \norm{q^{k+1}-\overline{q}}^2 + 2 \sum_{j=0}^k D_J^{p^j}(u^{j+1},u^j)   + \\
 & 2 \sum_{j=0}^k \langle p^{j+1} - K^* \overline{q}, u^{j+1} - \overline{u} \rangle + \sum_{j=0}^k \norm{q^{j+1}-q^j}^2  \\
 & \qquad \leq 2 D_J^{p^{0}}(\overline{u},u^{0}) + \norm{q^0 - \overline{q}}^2 .
\end{align*}
With the monotonicity of $\norm{q^{k+1}-q^k}$ from \eqref{eq:dissipation1} we hence deduce
\begin{equation}
\label{eq:qdecay}
	\norm{q^{k+1}-q^k} \leq \frac{C}{\sqrt{k+1}}
\end{equation}
for some constant $C$. 

Moreover, we deduce that $q^k$ is bounded, which implies the existence of a subsequence
$q^{k_\ell}$ converging weakly to some $\hat q$. Moreover, from \eqref{eq:dissipation2a} we conclude that $Ku^k$ is uniformly bounded, thus  
$$ J(u^{k}) = \langle K^* q^{k}, u^{k} \rangle = \langle q^{k}, K u^{k} \rangle \leq \norm{q^k} \norm{Ku^k} $$
is uniformly bounded as well. From the boundedness of the sublevel sets of $\frac{1}2\norm{K\cdot}^2+J$ we obtain the existence of a weak-star convergent subsequence of $u^k$, again denoted by $k_\ell$. 

Finally, since $q^{k+1}-q^k \rightarrow 0$, we conclude that the limits $q$ and $u$ of such converging subsequences satisfy $q = Wq + f- Ku$. 
Moreover, the weak convergence of $q^{k_\ell}$ yields
$$ \langle K_i^* q_i, v \rangle = \lim \langle K_i^* q_i^{k_\ell}, v \rangle \leq J(v) \qquad \forall~v \in {\cal U}.$$
If $K_i^*$ is compact, then the $K_i^* q_i^{k_\ell}$ converges strongly and hence
$$ \langle K_i^* q_i^{k_\ell}, u_i^{k_\ell} \rangle \rightarrow \langle K_i^* q_i, u_i \rangle. $$
Now the lower semicontinuity of $J$ implies for any $v_i \in {\cal U}$
$$ J(u_i) + \langle K_i^* q_i,v_i- u_i \rangle \leq \lim\inf J(u_i^{k_\ell}) + \langle K_i^* q_i^{k_\ell},v_i- u_i^{k_\ell} \rangle  \leq J(v_i),$$
hence $K_i^* q_i \in \partial J(u_i)$.
\end{proof}

From \eqref{eq:qdecay} and $\|W\| = 1$ we can immediately conclude
\begin{corollary}
\label{cor:convergencespeed}
Let the assumptions of Theorem \ref{cleandataconvergence} be satisfied. Then the color Bregman iteration converges with 
\begin{align}
\|r^{k}\| \leq \frac{C}{\sqrt{k}}
\end{align}
for some constant $C$.
\end{corollary}

Under the stronger condition that there is indeed a solution with subgradient $\overline{p}=\overline{p}_1 e_1$ one can easily give a refined statement: 
\begin{corollary}
\label{cor:convergenceToData}
Let in addition to the assumptions of Theorem \ref{cleandataconvergence} $\overline{u}$ satisfy $K\overline{u}=f$. Then $Ku^k \rightarrow f$ and $p^k$ converges weakly to $p_1 e_1$.  
\end{corollary}
\begin{proof}
The assertion follows from Theorem \ref{cleandataconvergence} combined with a summation of \eqref{eq:dissipation2a}.
\end{proof}

\begin{example}
 For regular Bregman iteration without any coupling of the subgradients we have $W=I$. Therefore, corollary \ref{cor:monoton1} tells us that $\|r^k\| = \|Ku^k - f\|$ is monotonically decreasing and corollary \ref{cor:convergencespeed} tells us that $\|Ku^k - f\| \leq \frac{C}{\sqrt{k}}$, which both is in agreement with the results of \cite{osh-bur-gol-xu-yin}. 
\end{example}
\subsection{Noisy Data} 
A key property of Bregman iterations for noisy data is the decrease of the Bregman distance to the clean image at least until the residual reaches the noise level. This yields the stability of the iteration and a strategy to choose a stopping criterion in dependence of the noise. 

In the case of the color Bregman iteration \eqref{eq:dissipation3} this will not work by comparing $r^{k+1}$ and $r$, but rather shifted versions, namely
\begin{equation}
	\sigma^k = r^{k+1} - (I-W) (q^k-q^{k+1}) = q^{k+1}-q^k,
\end{equation}
and 
\begin{equation}
	\overline{\sigma}^k = r - (I-W)(q^k - \overline{q}). 
\end{equation}

Indeed, from \eqref{eq:dissipation3} we can conclude   
\begin{equation}
 D_J^{p^{k+1}}(\overline{u},u^{k+1}) \leq D_J^{p^{k}}(\overline{u},u^{k}) 
\end{equation}
as long as 
\begin{equation}
	\Vert \sigma^k \Vert \geq \Vert \overline{\sigma}^{k} \Vert. \label{eq:discrepancy}
\end{equation}
At the initial iteration the comparison in \eqref{eq:discrepancy} is the same as in the classical discrepancy, since
$$ \sigma^0 = f - Ku^1, \qquad \overline{\sigma}^0 = f -K \overline{u}. $$
Moreover, we have shown above that $\Vert \sigma^k \Vert$ is decreasing during the iteration, which further indicates that there is an appropriate stopping index at which \eqref{eq:discrepancy} is violated for the first time. 

\begin{remark}
The above convergence analysis was based on formulation \eqref{eq:alternativeiteration}, which, as pointed out above, is only equivalent to \eqref{eq:primalupdate} in case that $WK^* = K^* W$. Note that this property does not only hold for denoising operators, but for all image reconstruction operators (such as blur or inpainting operators) which act on all color channels in the same way (e.g. the same blur operator for all color channels). For applications where the assumption $WK^* = K^* W$ does not hold, we have two types of iterative procedures, namely \eqref{eq:primalupdate} and \eqref{eq:alternativeiteration}, which could both be used, but potentially yield different results. For formulation \eqref{eq:alternativeiteration} the convergence theory still holds. The only limitation is that it does not necessarily have the interpretation of minimizing the Bregman distance between the channels anymore. 
\end{remark}

In the numerical results section we will present detailed results for the application of color Bregman iteration to TV denoising as well as some results on image deblurring and - as an exemplary problem - results for image inpainting for which the assumption $WK^*=K^*W$ does not hold. 

\section{Computational Results}
\label{sec:numerics}
In this section we will demonstrate the general behavior of the color Bregman TV as well as of the infimal convolution Bregman iteration model. In our numerical experiments of denoising 42 test images and a comparison to TV, TV with separate channel Bregman iteration, and vectorial TV, our proposed methods, show very promising results. Additional to the denoising results, we include detailed numerical experiments on different artificial test images showing the behavior of color Bregman iteration for images with edges pointing in different directions in the supplementary material. 

Before jumping into the numerical experiments, let us very briefly discuss the numerical implementation for both approaches, color Bregman iteration and infimal convolution Bregman iteration.
\subsection{Denoising}
\subsubsection{Numerical Implementation of Color Bregman TV Denoising}
In the case of TV denoising (using the simplified notation $J(u) = \|\nabla u\|_1$ for the total variation) we rewrite \eqref{eq:primalupdate} by defining the sum of all mixed residual $\tilde{p}_i^k$ with $\tilde{p}_i^0 = 0$ and obtain
\begin{align}
u_i^{k+1} &\in \text{arg}\min_{u_i \in {\cal U}} \left( \frac{1}2 \Vert u_i  - (f_i + \tilde{p}_i^k) \Vert^2 + \alpha \|\nabla u_i\|_1 \right),\\ 
\tilde{p}_i^{k+1} &= \sum_j w_{i,j} (f_j + \tilde{p}_i^k - u_j^{k+1}), \label{eq:colbregiter}
\end{align}
which is similar to the `adding-back-the-noise' formulation from \cite{osh-bur-gol-xu-yin}. 

For the above minimization in $u$, we apply the split Bregman method \cite{splitbregman} also known as the alternating directions method of multipliers (ADMM), see \cite{boyd11}: A new variable $d$ and an additional constraint $d = \nabla u$ is introduced, which is enforced using Bregman iteration. The minimization between $u$ and $d$ is done in an alternating fashion. This leads to an inner loop (ADMM algorithm to minimize for $u$) and an outer loop for the color Bregman iteration given by \eqref{eq:colbregiter}. The inner loop becomes
\begin{align*}
u_i & \leftarrow \text{arg}\min_{u_i } \left( \frac{1}2 \Vert u_i  - f^k_i \Vert^2 + \frac{\mu}2 \Vert \nabla u_i  - d_i + b_i \Vert^2  \right), \\
d_i  &\leftarrow \text{arg}\min_{d_i } \left( \frac{\mu}2 \Vert \nabla u_i  - d_i + b_i \Vert^2   + \alpha \|d_i\|_1 \right),\\
b_i &\leftarrow b_i + \nabla u_i -  d_i.
\end{align*}
Or, using the optimality conditions for $u$ and $d$ 
\begin{align*}
u_i &\leftarrow (I - \mu \Delta)^{-1} \left( f^k_i  + \mu  \text{div}(d_i - b_i)  \right), \\
d_i  &\leftarrow  \text{shrink}\left(\nabla u_i + b_i, \frac{\alpha}{\mu}\right),\\
b_i &\leftarrow b_i + \nabla u_i -  d_i.
\end{align*}

In our implementation we use an adaptive $\mu$ as well as a stopping criterion in the inner loop based on the primal and dual residual, see \cite{boyd11} for details. 

\subsubsection{Numerical Implementation of Infimal-Convolution Bregman TV}
The infimal convolution iteration is a little more challenging to minimize, since we have to find the decomposition of each $u^{k+1}_i$ into a part for which we measure the Bregman distance to $u_j^k$ and one part with a Bregman distance to $-u_j^k$ for each $j$. Again, we apply the ADMM algorithm. For the sake of avoiding too many indices, we omit additional indices indicating the iteration number of the minimization algorithm. We rewrite the minimization with respect to $u_i$ as the minimization of  
\begin{align*}
 & \frac{\lambda}{2} \|u_i - f_i\|^2 + w_{i,i} (\|d_i\|_1 - \langle q^k_i,d_i \rangle) + \frac{\mu}{2} \|\nabla u_i  - d_{i} + b_{i}\|^2  \\
  &+ \sum_{j\neq i} w_{i,j}(\|d_{j,+}\|_1 - \langle q^k_j, d_{j,+} \rangle + \|d_{j,-}\|_1 + \langle q^k_j, d_{j,-} \rangle) \\
  &+ \frac{\mu}{2} \sum_{j\neq i} (\|\nabla(u_i - \phi_j) - d_{j,+} + b_{j,+}\|^2 + \|\nabla \phi_j  - d_{j,-} + b_{j,-}\|^2)
\end{align*} 
with respect to $(u_i, (\phi_j, d_i, d_{j,+}, d_{j,-})_{j=1, ... c})$ and the constraints $\nabla u_i = d_i$, $\nabla (u_i - \phi_j) = d_{j,+}$ and $\nabla \phi_j = d_{j,-}$. 
We solve for $(u_i, \{\phi_j\})$, and $(d_i , \{d_{j,+}\}, \{d_{j,-}\})$ in an alternating fashion and obtain
\begin{align*}
0 &= \lambda (u_i -f_i) - \mu \text{div}\big(\nabla u_i - d_i + b_i +\sum_{j \neq i} (\nabla u_i - \nabla \phi_j - d_{j,+} + b_{j,+}) \big), \\
0 &= \mu ~ \text{div}(\nabla u_i - \nabla \phi_j - d_{j,+} + b_{j,+} - \nabla \phi_j + d_{j,-} - b_{j,-}) \qquad \forall j\neq i 
\end{align*}
as the optimality conditions for $(u_i, \{\phi_j\})$, which we can rewrite as
\begin{align}
(\lambda - \mu c \Delta) u_i + \sum_{j\neq i}\mu \Delta \phi_j &= \lambda f_i + \mu \text{div}(b_i - d_i + \sum_{j\neq i}(b_{j,+}- d_{j,+})), \label{eq:optimality_part1} \\
\Delta u_i - 2 \Delta \phi_j &= \text{div}(d_{j,+} - b_{j,+} +b_{j,-} - d_{j,-}) \qquad \forall j\neq i. \label{eq:optimality_part2}
\end{align}
Note that we can use the equations \eqref{eq:optimality_part2} in \eqref{eq:optimality_part1} to eliminate all $\phi_j$ such that solving for $u_i$ reduces to the linear equation
\begin{align*}
(\lambda - \frac{\mu c+1}{2} \Delta) u_i  &= \lambda f_i + \mu \text{div}(b_i - d_i + \frac{1}{2} \sum_{j\neq i}(b_{j,+}- d_{j,+} + b_{j,-} - d_{j,-})).
\end{align*}
After finding $u_i$, equations \eqref{eq:optimality_part2} yield a Poisson equation for each $\phi_j$. Both, the solving for $u_i$ and solving for the $\phi_j$ can be done very efficiently using the discrete cosine transform. 
The minimization with respect to $(d_i , \{d_{j,+}\}, \{d_{j,-}\})$ yields simple shrinkage formulas for each variable:
\begin{align*}
d_i &= \text{shrink}\left(\nabla u_i + b_i + \frac{w_{i,i}}{\mu} q_i^k, \frac{w_{i,i}}{\mu}\right), \\
d_{j,+} &= \text{shrink}\left(\nabla (u_i - \phi_j) + b_{j,+} + \frac{w_{i,j}}{\mu} q_j^k, \frac{w_{i,j}}{\mu}\right), \\
d_{j,-} &= \text{shrink}\left(\nabla \phi_j + b_{j,-} - \frac{w_{i,j}}{\mu} q_j^k, \frac{w_{i,j}}{\mu}\right) .
\end{align*}
Finally, we update $b_i$, $b_{j,+}$ and $b_{j,-}$. The formulas for the updates are the typical ADMM updates and are left out here for the sake of brevity. After convergence we use
$$ q_i^{k+1} = q_i^k + \frac{\mu}{w_{i,i}} b_i$$
to compute the $q^{k+1}$ for the next iteration. 
\subsubsection{Denoising Results on Artificial Images}
We presented detailed discussions about the behavior of the color Bregman as well as the infimal convolution Bregman approach regarding their properties in the cases where the data channels do or do not have a common subgradient. As shown above, the property of having a common subgradient translates to the edges being at the same positions and pointing into the same direction when using TV regularization. To closely investigate the effect of this property on color Bregman and infimal convolution iteration, we generated several simple test images with two nested, differently colored squares, such that the jumps sometimes point in the same and sometimes point in opposite directions. We ran both of the proposed methods for 20 iterations and illustrate the temporal evolution of the iteration in a sequence of images in the accompanying supplementary material. A more detailed description of the corresponding experimental setup can also be found in the supplementary material. 

\subsubsection{Denoising Results on Natural Color Images}

Despite the illustration of the behavior of the proposed methods on the artificial test images, the success of the proposed methods in real applications will depend on how well the assumption of common edge sets and edges pointing in the same direction is met in practice. Therefore, we will evaluate the denoising results on a set of real and realistic images. 

For our numerical experiments we used the 24 images of the Kodak image data set (cf.  {http://r0k.us/ graphics/kodak/}), a representative set of color images. We added Gaussian noise with a standard deviation of $\sigma=0.05$ (for images on a scale from $0$ to $1$) to obtain noisy color images and applied a channel-by-channel TV, regular Bregman iteration, infimal convolution Bregman iteration and color Bregman iteration to the noisy images and additionally compare our results to the vectorial total variation (VTV) approach proposed by Bresson and Chan in  \cite{Bre08}, for which the Matlab code is available on the author's website \textit{http://www.cs.cityu.edu.hk/~xbresson/codes.html}. 
The parameter $\mu$ for the Bregman iteration methods and the parameter $\lambda$ for the VTV method were optimized using the first five images of each data set and then these parameters were used for denoising the complete image data set. 
For the color Bregman TV iteration we set all weights to one third, $w_{i,j} = 1/3$, since this choice gave the best reconstruction quality in our numerical denoising experiments.

The mean peak signal-to-noise ratio (PSNR) values over all 24 Kodak images are shown in Tab. \ref{tab:meanValues}. We included the TV and Bregman iteration based TV results using the anisotropic as well as the isotropic definition of the TV. As we can see, isotropic color Bregman iteration yields the highest and infimal convolution Bregman iteration the second highest PSNR values. The mean PSNR using isotropic color Bregman iteration is about 1.35 dB higher than the PSNR using classical TV. In comparison to VTV, the infimal convolution approach yields a small improvement while color Bregman iteration can further improve the infimal convolution results by more than 0.5dB. 

Figure \ref{fig:psnrSorted} shows the PSNR results of all methods on all 24 Kodak images as a line plot. For the sake of readability the results are sorted in descending order based on the color Bregman PSNR value. As we can see the color Bregman iteration results show the highest PSNR for all the 24 images. The infimal convolution approach outperforms VTV in most cases but remains clearly behind the color Bregman iteration results. Thus, for our test data, the stricter assumption of the edges pointing into the same direction seems to help much more with the suppression of the noise than it harms at parts where the assumption might be violated. 

	We additionally used the McMaster (McM) data set (cf. {http://www4.comp.polyu.edu.hk/~cslzhang/ CDM\_Dataset.htm}). Similarly to the Kodak data results, the mean PSNR values in Tab. \ref{tab:meanValuesMcM} show that isotropic color Bregman iteration yields the best results in terms of the PSNR. Interestingly, on this dataset the infimal convolution iteration yields almost equally good results. This may be due to the extremely strong colors in the images from the McM data set, which may lead to a large amount of edges pointing in different directions. The line plot in figure~\ref{fig:psnrSorted} shows that color Bregman iteration outperforms VTV on all images with one exception: image number 17, which shows a highly saturated image of a green and red flower. As the edges are mostly pointing into different directions, infimal convolution shows the best result on this image.
	
	A visual comparison of the results obtained on Kodak image 23 (see figure \ref{fig:cropsKodak23}) shows the advantage of obtaining higher contrast images with color Bregman iteration. The contrast and detail recovery is highest in the color Bregman results. Notice that a higher noise level of $\sigma=0.15$ was chosen in this example to better show the differences of the method results. The reference and the noisy image crop of the Kodak image 23 are shown in figure \ref{fig:23ref} and \ref{fig:23noisy} while figures \ref{fig:23TViso} - \ref{fig:23ColBreg} show the results of classical channel-by-channel TV, vectorial TV and the two proposed methods.
In figure~\ref{fig:23InfConf} and figure~\ref{fig:23ColBreg} the results of the proposed infimal convolution and the color Bregman iteration are shown. Around the eye of the bird we observe higher contrast in the infimal convolution and the color Bregman iteration results compared to the TV results. The black color in the circle around the eye can be recovered completely using the color Bregman iteration. The contrast using infimal convolution Bregman iteration is still higher than the contrast of VTV (figure~\ref{fig:23VecTV}), thus the mentioned circle is darker in the infimal convolution result than in the VTV result. The color Bregman iteration is the only method that brings back details (dark spots) on the blue wing of the bird.

\begin{figure}[H]
  \centering
  \subfloat{
	    \centering
	  \includegraphics[width=0.531\textwidth, natwidth=164.7mm,natheight=112.9mm]{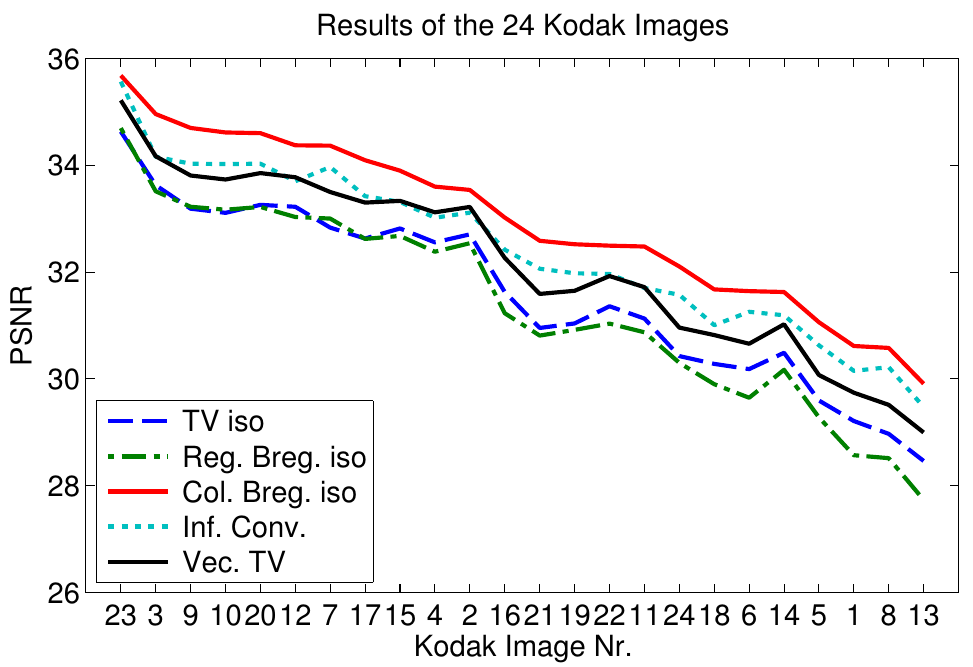}
  }
  \subfloat{
	    \centering
	  \includegraphics[width=0.426\textwidth, natwidth=132.3mm,natheight=112.9mm]{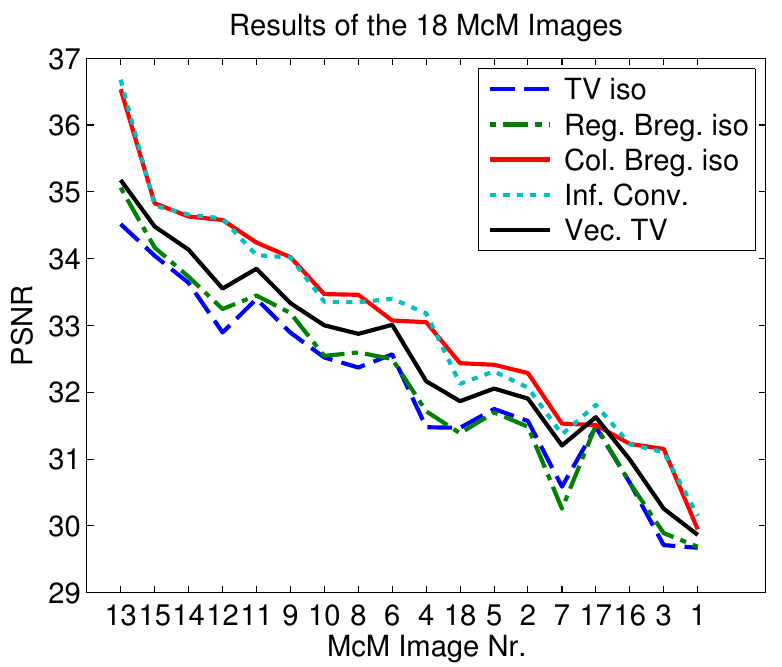}
  }
  \caption[PSNR results sorted]{PSNR results for the 24 images of the Kodak image set on the left and for the McM data set on the right, sorted in descending order ($\sigma=0.05$).}
  \label{fig:psnrSorted}
\end{figure}
\begin{table}[H]
  \centering
  \begin{tabular}{ll}
    Method & PSNR \\
    \hline
    \hline
    Anisotropic TV &31.60\\
    Isotropic TV  &31.59   \\
    Anisotropic channel-by-channel Bregman iteration &   31.07\\
    Isotropic channel-by-channel Bregman iteration & 31.38   \\
    Anisotropic color Bregman iteration & 32.63 \\
    Isotropic color Bregman iteration & \bf{32.95}    \\
    Isotropic Infimal convolution Bregman iteration & 32.41   \\
    Vectorial TV & 32.16   \\

  \end{tabular}
  \caption[Mean value of the 24 PSNR results]{Mean PSNR results of the 24 Kodak images with AWGN ($\sigma=0.05$). }
  \label{tab:meanValues}
\end{table}
\begin{table}[H]
  \centering
  \begin{tabular}{ll}
    Method & PSNR \\
    \hline
    \hline
    Anisotropic TV &31.99 \\
    Isotropic TV  &32.07   \\
    Anisotropic channel-by-channel Bregman iteration &   31.99\\
    Isotropic channel-by-channel Bregman iteration & 32.15   \\
    Anisotropic color Bregman iteration & 32.69 \\
    Isotropic color Bregman iteration & \bf{33.02}    \\
    Isotropic Infimal convolution Bregman iteration & 33.01   \\
    Vectorial TV & 32.52   \\

  \end{tabular}
  \caption[Mean value of the 18 PSNR results]{Mean PSNR results of the 18 McM images with AWGN ($\sigma=0.05$). }
  \label{tab:meanValuesMcM}
\end{table}

\begin{figure}[H]
\captionsetup[subfigure]{labelformat=empty}
  \centering
  \subfloat{
   \includegraphics[width=0.29\textwidth, natwidth=460,natheight=440]{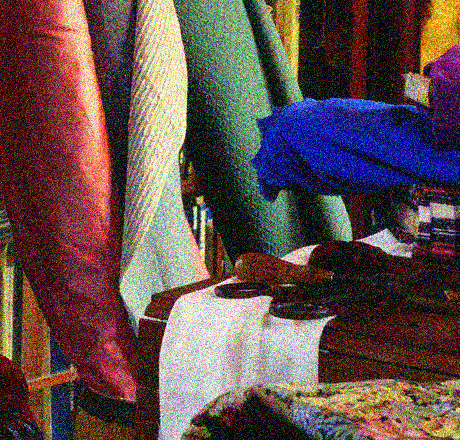}
   }
     \subfloat{
   \includegraphics[width=0.29\textwidth, natwidth=460,natheight=440]{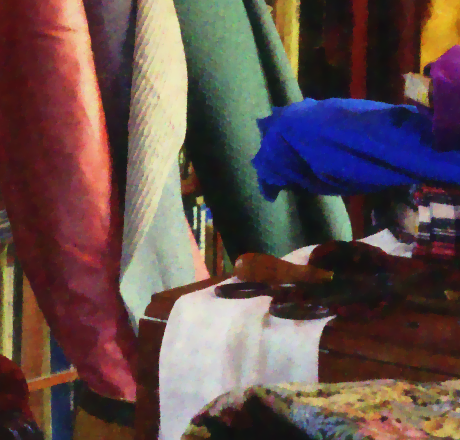}
   }
  \subfloat{
   \includegraphics[width=0.29\textwidth, natwidth=460,natheight=440]{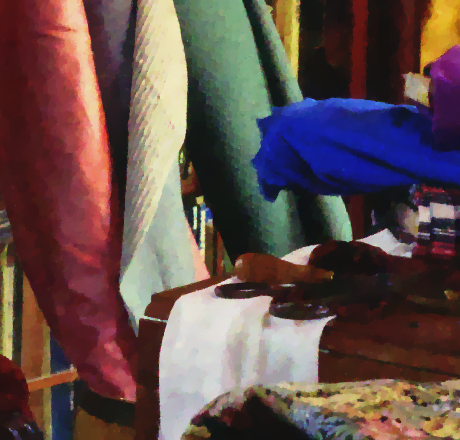}
   } \\   
  \subfloat{
   \includegraphics[width=0.29\textwidth, natwidth=460,natheight=440]{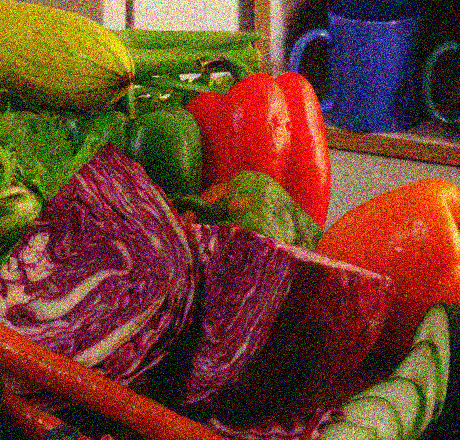}
   }
     \subfloat{
   \includegraphics[width=0.29\textwidth, natwidth=460,natheight=440]{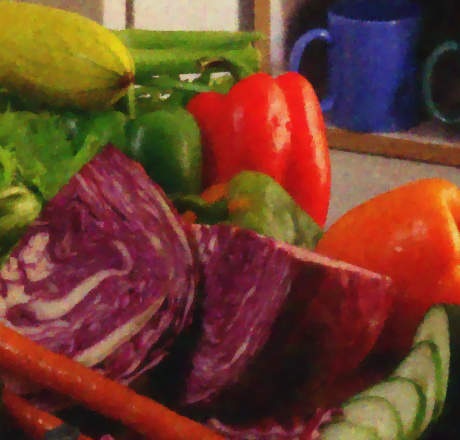}
   }
  \subfloat{
   \includegraphics[width=0.29\textwidth, natwidth=460,natheight=440]{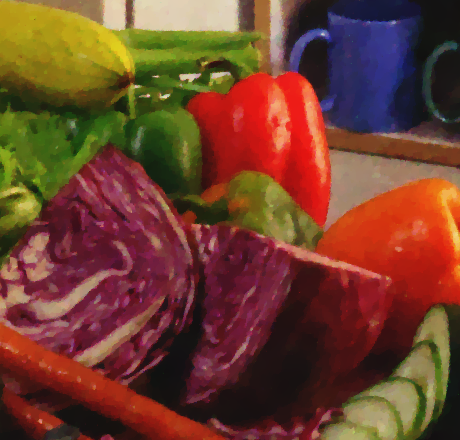}
   } \\ 
  \subfloat{
   \includegraphics[width=0.29\textwidth, natwidth=460,natheight=440]{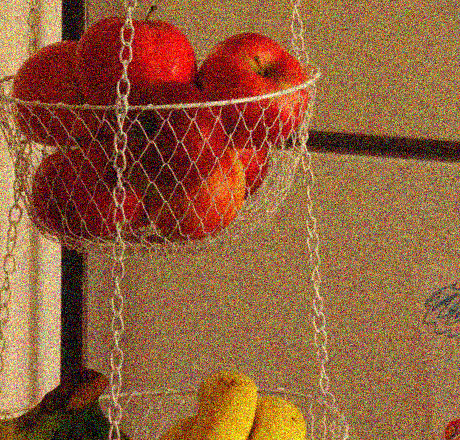}
   }
     \subfloat{
   \includegraphics[width=0.29\textwidth, natwidth=460,natheight=440]{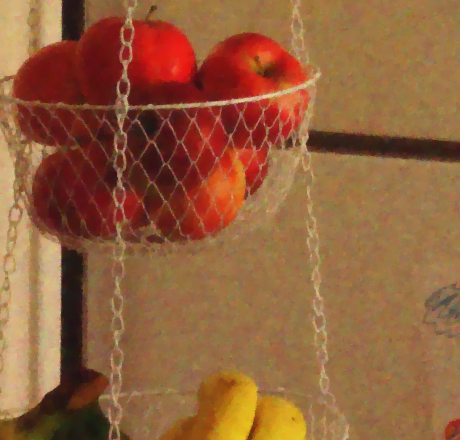}
   }
  \subfloat{
   \includegraphics[width=0.29\textwidth, natwidth=460,natheight=440]{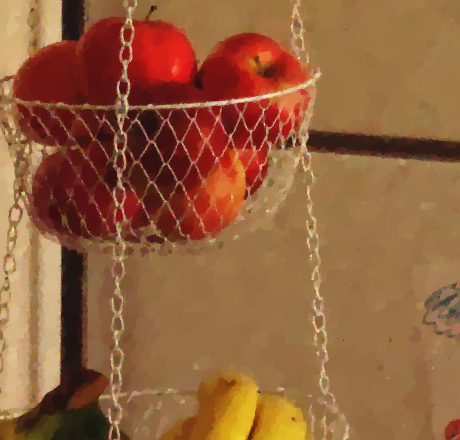}
   } \\ 
  \subfloat[Noisy ($\sigma=0.15$)]{
   \includegraphics[width=0.29\textwidth, natwidth=460,natheight=440]{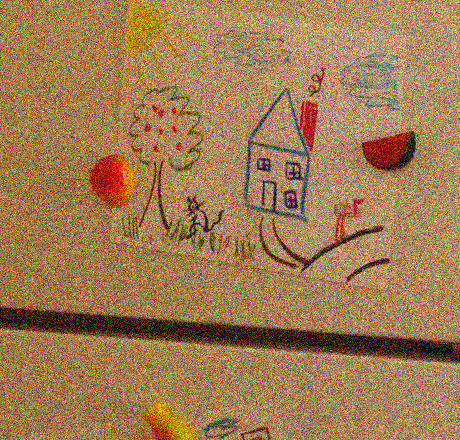}
   \label{fig:mcm2noisy}
   }
     \subfloat[Vectorial TV \cite{Bre08}]{
   \includegraphics[width=0.29\textwidth, natwidth=460,natheight=440]{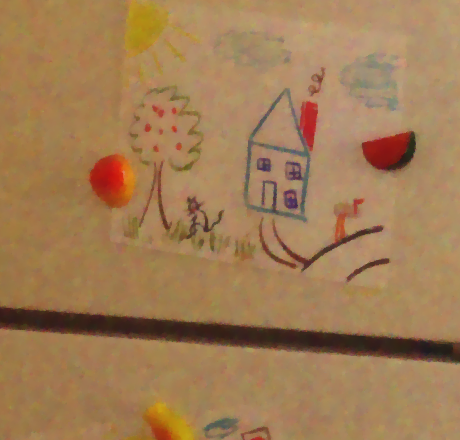}
   \label{fig:mcm2vtv}
   }
  \subfloat[Color Bregman TV]{
   \includegraphics[width=0.29\textwidth, natwidth=460,natheight=440]{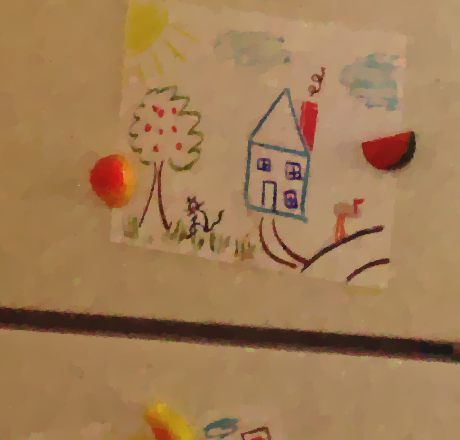}
   \label{fig:mcm2cbi}
   } 
    \caption[Denoising results on the McM data set]{Denoising results on the McM data set, from top to bottom: image number 2,10,12 and 13, from left to right: noisy image with $\sigma=0.15$, denoised image using Vectorial TV \cite{Bre08}, denoised image using the proposed color Bregman iteration.} 
    \label{fig:mcmImages}
\end{figure}

\begin{figure}[ht]
  \subfloat[Reference]{
   \includegraphics[width=0.31\textwidth, natwidth=236,natheight=236]{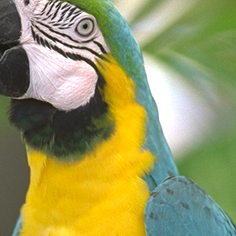}
   \label{fig:23ref}
   }
     \subfloat[Noisy ($\sigma=0.15$)]{
   \includegraphics[width=0.31\textwidth, natwidth=236,natheight=236]{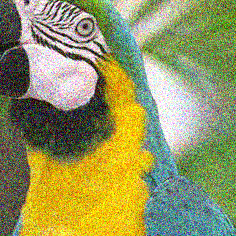}
   \label{fig:23noisy}
   }
  \subfloat[Isotropic TV]{
   \includegraphics[width=0.31\textwidth, natwidth=236,natheight=236]{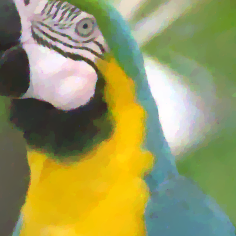}
   \label{fig:23TViso}
   } \\ 
   
   \subfloat[Vectorial TV \cite{Bre08}]{
   \includegraphics[width=0.31\textwidth, natwidth=236,natheight=236]{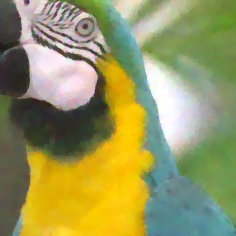}
   \label{fig:23VecTV}
   }
     \subfloat[Infimal Convolution Bregman]{
   \includegraphics[width=0.31\textwidth, natwidth=236,natheight=236]{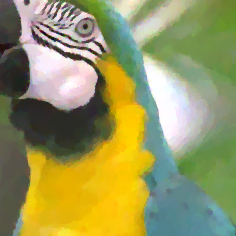}
   \label{fig:23InfConf}
   }
     \subfloat[Color Bregman TV]{
   \includegraphics[width=0.31\textwidth, natwidth=236,natheight=236]{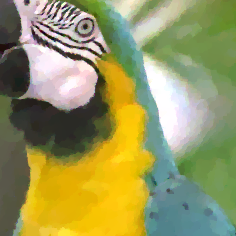}
   \label{fig:23ColBreg}
   }
    \caption[Crop of the Kodak Image no. 23]{Crop of the Kodak image no. 23. } 
    \label{fig:cropsKodak23}
\end{figure}

A visually unpleasant effect of TV denoising is the staircasing effect, which is especially visible in the smooth white-to-green transition in the right part of the image. Note that the proposed scheme of color Bregman iterations does not change structural properties of the regularizer. Therefore, we can see staircasing in the color Bregman TV reconstruction as well. Even more so, there is no loss of contrast for a given edge set which leads to the staircasing appearing even more evident. However, one big advantage of the proposed Bregman iteration is, that it immediately generalizes to any kind of regularizer. For illustration purposes we applied the color Bregman iteration for denoising the same parrot image shown in figure \ref{fig:cropsKodak23} to the total generalized variation (TGV) proposed in \cite{Bre10}. As we can see in figure \ref{fig:GTVParrot}, the eye and structures in the face of the parrot are reconstructed nicely without a loss of contrast while the staircasing we observed in the TV case is suppressed.

For further visual comparison purposes, figure~\ref{fig:mcmImages} shows four noisy images obtained from the McM data set and a comparison of the VTV and color Bregman TV results. In all images we observe a higher contrast and better preservation of the details in the color Bregman results: In the towel shown in image number 2 the weaving structure of the towel's fabric is clearly more visible in the color Bregman TV result than in the VTV result. The same holds for the red cabbage in image number 10 and for the wire net holding the apples in image number 12. Generally, we see that TV regularization is best suited for piecewise constant images. Hence, we obtain the best denoising results on MCM image 13, which has a plain white background. In comparison to the color Bregman TV image, the corresponding VTV image looks surprisingly noisy although we used some brute force optimization of the PSNR to find the best regularization parameter for each denoising method. 

\begin{figure}[H]
\begin{center}
\subfloat[TGV denoising]{
   \includegraphics[width=0.31\textwidth, natwidth=236,natheight=236]{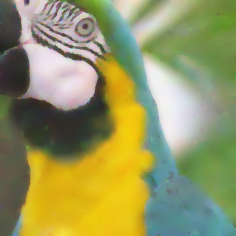}
   }
   \subfloat[Separate channel Bregman TGV]{
   \includegraphics[width=0.31\textwidth, natwidth=236,natheight=236]{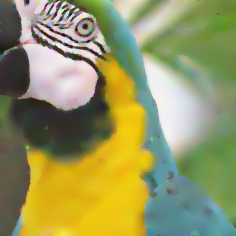}
   }
   \subfloat[Color Bregman TGV]{
   \includegraphics[width=0.31\textwidth, natwidth=236,natheight=236]{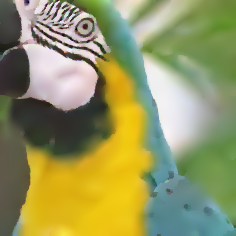}
   }
   \caption{Comparison of the results using TGV$^2$, separate channel Bregman iteration with TGV$^2$, and color Bregman iteration with TGV$^2$. For the Bregman iterative procedures we used $0.7$ and $1.12$ as the weights for the first and second order derivatives respectively. For the plain TGV regularization we used $0.14$ and $0.224$. As we can see, the color Bregman iteration with TGV as a regularization function profits from both advantages - the structural improvements of TGV over TV as well as of the improved contrast and cross-channel regularization of the color Bregman iteration: The staircasing is eliminated, the contrast of the parrots face is much better in the Bregman iterative methods than in the plain TGV denoising and the color Bregman TGV has edges which are better aligned and show less colored artifacts.}
   \label{fig:GTVParrot}
\end{center}
\vspace{-0.5cm}
\end{figure}

\subsection{Deblurring}

In order to test the behaviour of the color Bregman iteration for deblurring, we use a setup with a Gaussian blur (equal on all channels) and subsequent additive Gaussian noise, i.e.
\begin{equation}
	f(x) = \left(\int G(x-y) u_i(y) ~dy\right)_{i=1,2,3} + \sigma n(x), 
\end{equation}
for each pixel, where $n  \sim N(0,I_3)$ is a three-dimensional unit normal random variable. 
Note that since we use the same convolution operator on each channel, the formulations  \eqref{eq:primalupdate} and \eqref{eq:alternativeiteration} are equivalent. Hence we solve the first by applying a standard forward backward splitting method together with an ADMM solver for the backward problem as described in \cite{Saw11,Saw14}.

The results are illustrated again for Kodak image 23 in figure \ref{fig:deblurringKodak23}, where we used $\sigma = 0.025$ for $u$ taking values between $0$ and $1$. In the reconstruction we use the regularization parameter $\alpha=0.0025$. The standard Bregman iteration (displayed in subfigure \ref{fig:deblurringKodak23}(d)) one obtains by choosing the weights of color Bregman iteration according to $w_{i,j}=\delta_{ij}$ achieved a SNR of 23.10 dB. The color Bregman iteration with equal weights $w_{i,j}=\frac{1}3$ as shown in subfigure \ref{fig:deblurringKodak23}(e) improves the result to a SNR of 23.36 dB. Besides the improved SNR, there is a clear visual improvement e.g. in the face region. Surprisingly and different from our results in the denoising case, the choice of equal weights is not optimal in this example. An improved SNR of 23.39 dB is obtained by using a higher weight for the Bregman distance to the channel that is currently reconstructed, i.e.  $w_{i,j}=\frac{1}4(1+\delta_{ij})$. Looking at the corresponding reconstruction in subfigure \ref{fig:deblurringKodak23}(f) it is, however, difficult to find visual differences to the case of equal weights. Since the results with the best SNR are not always the visually most appealing ones, we included the results after different numbers of Bregman iterations in the supplementary material.
\begin{figure}[H]
\begin{center}
\begin{tabular}{>{\centering\arraybackslash}m{0.3\textwidth}  >{\centering\arraybackslash}m{0.3\textwidth}>{\centering\arraybackslash}m{0.3\textwidth} }
 \includegraphics[width=0.3\textwidth, natwidth=236,natheight=236]{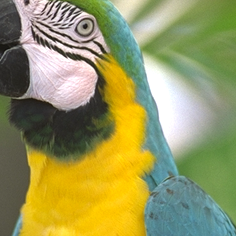} &
  \includegraphics[width=0.3\textwidth, natwidth=236,natheight=236]{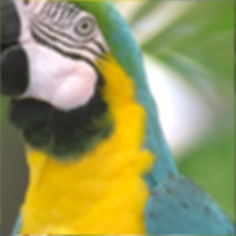} &
  \includegraphics[width=0.3\textwidth, natwidth=236,natheight=236]{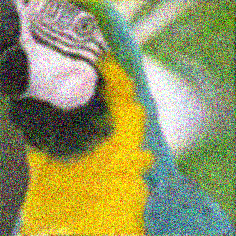} \\
\footnotesize (a) Ground truth  & \footnotesize(b) Blurred image  & \footnotesize(c) Blurred and noisy image \\
    & &  \\
   \includegraphics[width=0.3\textwidth, natwidth=236,natheight=236]{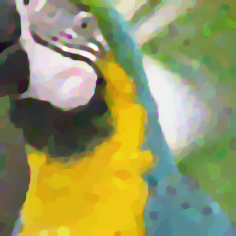} &
    \includegraphics[width=0.3\textwidth, natwidth=236,natheight=236]{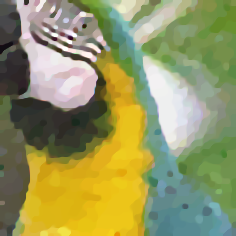} & 
     \includegraphics[width=0.3\textwidth, natwidth=236,natheight=236]{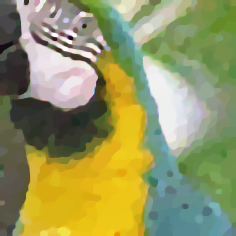} \\
     \footnotesize(d) Bregman iteration with weights \mbox{$w_{i,j}=\delta_{ij}$}&\footnotesize (e) Color Bregman TV with weights $w_{i,j}=\frac{1}3$ & \footnotesize(f) Color Bregman TV with weights $w_{i,j}=\frac{1}4(1+\delta_{ij})$
\end{tabular}
    \caption[Deblurring test on the Kodak image no. 23]{Deblurring test on the Kodak image no. 23. } 
    \label{fig:deblurringKodak23}
    
\end{center}
\end{figure}

\subsection{Inpainting}

For illustration purposes and as a proof of concept that our general idea/the framework of color Bregman iteration is applicable to a wide range of image reconstruction problems, let us investigate its behavior in the case of color image inpainting. 
\begin{figure}[h]
	\centering
  \subfloat[Ground truth]{
   \includegraphics[width=0.31\textwidth, natwidth=256,natheight=100]{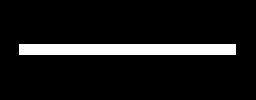}
   }
   \subfloat[Red inpainting domain]{
   \includegraphics[width=0.31\textwidth, natwidth=256,natheight=100]{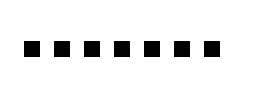}
   }
   \subfloat[Green inpainting domain]{
   \includegraphics[width=0.31\textwidth, natwidth=256,natheight=100]{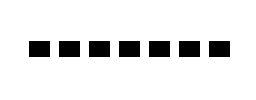}
   }\\
	\subfloat[Blue inpainting domain]{
   \includegraphics[width=0.31\textwidth, natwidth=256,natheight=100]{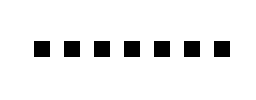}
   }  
   \subfloat[Channel-by-channel TV]{
   \includegraphics[width=0.31\textwidth, natwidth=256,natheight=100]{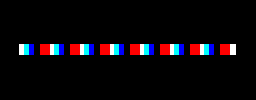}
   }
   \subfloat[Vectorial TV]{
   \includegraphics[width=0.31\textwidth, natwidth=256,natheight=100]{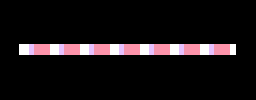}
   }  \\
     \subfloat[Color Bregman TV with \eqref{eq:alternativeiteration}]{
   \includegraphics[width=0.31\textwidth, natwidth=256,natheight=100]{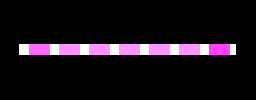}
   }     
  \subfloat[Color Bregman TV with \eqref{eq:primalupdate}]{
   \includegraphics[width=0.31\textwidth, natwidth=236,natheight=100]{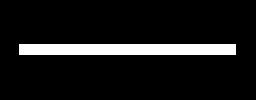}
   } 
    \caption{Inpainting results of color Bregman TV for a toy problem. Closing all gaps works best when applying the color Bregman iteration.} 
    \label{fig:colorBregInpainting}
    \vspace{-0.3cm}
\end{figure}

There are two possible versions of TV inpainting - one where a hard equality constraint is imposed on the data outside the inpainting domain and one where the given measurements are denoised rather than matched exactly. Note that channel-by-channel Bregman iteration starts with a solution of the latter problem and iteratively converges to a solution of the first formulation. For color Bregman iteration this is not necessarily the case anymore. We will focus on the denoising type of formulation in this section.

Let us briefly mention that Bregman distances have previously been used for image inpainting tasks (cf. \cite{Bal03,moeller12}), however, not in the iterative procedure proposed in this paper. Generally, other types of color correlation have been proposed in literature, for instance the interpolation of chrominance values or interpolation in different color spaces (cf.  \cite{ZhangWu11, Laroche94}) or changes of color space, e.g. \cite{Con12}. In our framework a change of color space of course clearly affects the assumption of edges going into the same direction such that one might have to combine them with the infimal convolution version of color Bregman iteration.

In this section we simply take a small toy problem to illustrate the effect of straight forward color Bregman iteration. It is interesting to see that in case the inpainting domains are different for different color channels, the iterative schemes \eqref{eq:primalupdate} and \eqref{eq:alternativeiteration} will not be equivalent any more. As an example, let us take a classical inpainting problem, where the channel-by-channel total variation is well known to fail: As proven by Meyer in \cite{meyer}, the total variation will connect two bars with an unknown region in the middle if and only if the length of the region to be inpainted is as most as large as the width of the bars. However, if we assume that the unknown regions of the different color channels differ, one can use the information (assuming a common edge set) to improve the inpainting results. 
Figure \ref{fig:colorBregInpainting} shows a toy example with the ground truth, three inpainting domains for the three color channels red, green, and blue, where black regions are unknown, and the results of channel-by-channel TV, vectorial TV, and color Bregman TV using formulation \eqref{eq:primalupdate} and \eqref{eq:alternativeiteration}. Note that the unknown gaps in the inpainting domains are not only too large for the separate channel to connect them, they are also overlapping such that there are some parts where seemingly no information is available. Numerically, we again use an ADMM approach (with an additional splitting) and perform a Bregman iteration after every 50 inner iterations. 

We can see that - as expected - the channel-by-channel TV approach is basically unable to close any gaps. The vectorial TV approach is able to close some gaps and additionally leads to smaller intensity jumps in the remaining gaps. However, structurally vectorial TV cannot reconstruct the true solution. It is very interesting to see that the results of color Bregman TV using formulation \eqref{eq:primalupdate} and \eqref{eq:alternativeiteration} actually differ. While the formulation based on \eqref{eq:alternativeiteration} has problems to close all the gaps (similar to the vectorial TV), the formulation \eqref{eq:primalupdate} led to a perfect reconstruction with all gaps closed. This indicates that our original motivation of using the Bregman distance between different color channels as a regularization might be preferable in cases where \eqref{eq:primalupdate} and \eqref{eq:alternativeiteration} are not equivalent.

\section{Conclusions and Further Extensions}
\label{sec:conclusions}
In this paper we proposed an iterative procedure for the reconstruction of multichannel data using Bregman distances, which encourages a common subgradient of all channels. In the case of TV regularization it leads to the image channels sharing a common edge set as well as a common edge direction. The latter effect can be avoided by using the infimal convolution of Bregman distances. We proved the convergence of the new iterative procedure, analyzed stationary points and demonstrated its effectiveness with numerical experiments on color image denoising, inpainting and deblurring.  

Since the idea of coupling subgradients with a weight matrix $W$ in a Bregman iteration is a completely novel idea, there are a number of open questions which go beyond the scope of this paper and will be subject to future research. Nevertheless, we would like to point out which open questions and further ideas we find particularly interesting.  
\begin{itemize}
\item First of all, the infimal convolution iteration is lacking a rigorous analysis such as a convergence proof, convergence speed, or stationary points of the iteration.  

\item Secondly, the above color Bregman iteration needs to be tested in detail in different settings (e.g. denoising, demosaicking, inpainting, or deconvolution) and compared to state of the art methods not related to Bregman iteration, such as variants of non-local means or dictionary learning approaches. As we have pointed out in section \ref{sec:primaldual} there are two different types of color Bregman iterations in the cases where $K^*$ and $W$ do not commute. While in the prominent case of denoising the commutation is obvious since $K=I$, an operator acting differently on different color channels, e.g. the operator arising in demosaicking, will not have this property. It will be interesting to investigate the differences between the two different Bregman iterations in this case.  

\item Thirdly, the application of our approach to multimodal imaging data will be interesting. Notice that the proposed Bregman iteration only assumes a relation between the edge sets of the image channels and not between the image values such that joint reconstruction of very different types of images is possible with each channel profiting from the information in the other channels. 
   
\item  Finally, note that we formulated the above Bregman iteration as 
$ q^{k+1 } = Wq^k + \lambda (f- Ku^{k+1}), $
where the most important parts of the convergence analysis are based on $W$ and $I-W$ being positive semi-definite and $\|W\| \leq 1 $. At least in the discrete setting, one could consider to not only couple different image channels, but also couple patches within each channel in a non-local fashion. Similar to non-local means (or non-local TV, cf. \cite{Gil09, Saw11}) one could determine similarity weights between image patches and use these weights in the operator/matrix $W$ to couple the parts of the subgradient of different patches. Notice that this technique could find significantly more similar patches than non-local means since not the patches themselves, but only their edge sets have to be similar: A red patch with a jump could help to denoise a green patch with a jump. This idea will also be subject to future research. 
\end{itemize}
\section{Acknowledgements}
The authors would like to thank Martin Benning for providing efficient MATLAB code for the inversion of the operator $I - \gamma \Delta$. This work was partly supported by the German Ministery for Science and Education (BMBF) via the project HYPERMATH. MB further acknowledges support from the German Science Foundation (DFG) via projects  BU 2327/6-1 and BU 2327/8-1.

\nocite{De85}
\bibliographystyle{siam}
\bibliography{references}

\end{document}